\documentclass[12pt]{article}

\usepackage{amsfonts}
\usepackage{amsmath}
\usepackage{amsthm}
\usepackage{hyperref}

\newcommand{\ds}{\displaystyle}
\numberwithin{equation}{section}

\def\ds{\displaystyle}
\newtheorem{theorem}{Theorem}[section]

\newtheorem{proposition}[theorem]{Proposition}

\newtheorem{corollary}[theorem]{Corollary}

\newtheorem{conjecture}{Conjecture}[section]
\usepackage[toc,page]{appendix}
\usepackage{subfigure}

\usepackage{amsmath, amssymb, graphics, setspace}

\newcommand{\mathsym}[1]{{}}
\newcommand{\unicode}[1]{{}}

\begin{document}

\title{Multiple orthogonal polynomials associated with an exponential cubic weight}
\author{Walter Van Assche\footnotemark[1] \quad Galina Filipuk\footnotemark[2]
\quad Lun Zhang\footnotemark[1]\ \footnotemark[3]}
\date{\today}

\maketitle
\renewcommand{\thefootnote}{\fnsymbol{footnote}}
\footnotetext[1]{Department of Mathematics, KU Leuven,
Celestijnenlaan 200B box 2400, BE-3001 Leuven, Belgium. E-mail: Walter.VanAssche@wis.kuleuven.be}
\footnotetext[2]{Faculty of Mathematics, Informatics and Mechanics,
University of Warsaw, Banacha 2, Warsaw, 02-097, Poland. E-mail: filipuk@mimuw.edu.pl}
\footnotetext[3]{School of Mathematical Sciences and Shanghai Key Laboratory for Contemporary Applied Mathematics, Fudan University, Shanghai 200433, People's Republic of China. E-mail: lunzhang@fudan.edu.cn}

\begin{abstract}
We consider  multiple orthogonal polynomials associated with the
exponential cubic weight $e^{-x^3}$ over two contours in the complex
plane.
We study the basic properties of these polynomials, including the
Rodrigues formula and nearest-neighbor recurrence relations. It
turns out that the recurrence coefficients are related to a discrete
Painlev\'{e} equation. The asymptotics of the recurrence
coefficients, the ratio of the diagonal multiple orthogonal polynomials and the (scaled) zeros of these polynomials are also investigated.

\vspace{2mm} \textbf{Keywords:} multiple orthogonal polynomials,
exponential cubic weight, Rodrigues formula, nearest-neighbor
recurrence relations, string equations, discrete Painlev\'{e}
equation, zeros, asymptotics
\end{abstract}

\section{Introduction and statement of the results}

\subsection{Orthogonal polynomials associated with an exponential cubic weight}
\label{sec:op}

A sequence of non-constant monic polynomials $\{p_n\}$ with $\deg
p_n\leq n$ is said to be orthogonal with respect to the exponential
cubic weight $e^{-x^3}$ if
\begin{equation}\label{eq:orthogonality}
\int_{\Gamma}p_{n}(x)x^k e^{-x^3}d x =0, \qquad k=0,1,\ldots,n-1,
\end{equation}
where the contour $\Gamma$ is chosen such that the above integral
converges. These polynomials satisfy the three-term recurrence
relation
\begin{equation} \label{recurrence}
x p_n(x) = p_{n+1}(x) + \beta_n p_n(x) + \gamma_n^2 p_{n-1}(x),
\end{equation}
where
\begin{equation}\label{recu:monic}
\beta_n=\frac{\int_\Gamma x p_n^2(x)e^{-x^3} dx}{\int_\Gamma
p_n^2(x)e^{-x^3} dx},\qquad \gamma_n^2=\frac{\int_\Gamma x
p_n(x)p_{n-1}(x)e^{-x^3} dx}{\int_\Gamma p_{n-1}^2(x)e^{-x^3} dx},
\end{equation}
and the initial condition is taken to be $\gamma_0^2 p_{-1}=0$. It
is shown by A. Magnus \cite{Magnus2} that the recurrence
coefficients $\beta_n$ and $\gamma_n^2$ satisfy the ``string''
equations
\begin{align}
    \gamma_{n+1}^2+\beta_n^2+\gamma_n^2=0,  \label{eq:string 1}\\
    3\gamma_n^2(\beta_{n-1}+\beta_n)=n. \label{eq:string 2}
\end{align}
For the convenience of the reader, we derive the string equations
using ladder operators for orthogonal polynomials in the Appendix.
Some variants of orthogonal polynomials associated with the
exponential cubic weight have recently been studied in the context
of numerical analysis \cite{DHK} and random matrix theory \cite{BD}.

\begin{figure}[ht]
\begin{center}
\rotatebox{270}{\resizebox{!}{3in}{\includegraphics{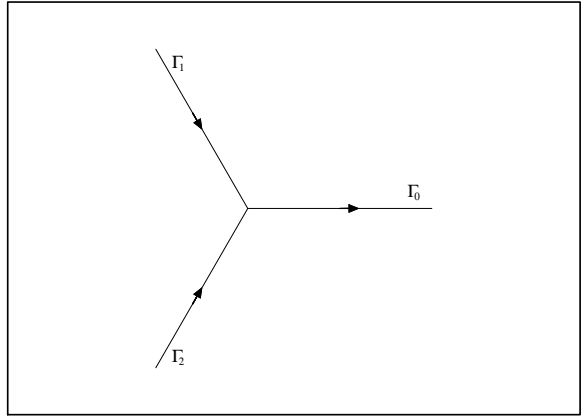}}}
\caption{The three rays $\Gamma_0, \Gamma_1, \Gamma_2$}
\label{fig:0}
\end{center}
\end{figure}

For our purpose, we are concerned with the polynomials for specific
contours $\Gamma$. Consider the three rays (see Figure \ref{fig:0})
\begin{equation}\label{def:contour}
\Gamma_k=\{z\in\mathbb{C}: \arg z=\omega^k \}, \qquad k=0,1,2,
\end{equation}
where $\omega=e^{2\pi i/3}$ is the primitive third root of unity
 and the orientations are all taken
from left to right. Clearly, the integral \eqref{eq:orthogonality}
is well-defined for each $\Gamma_k$. We shall denote by
$p_{n}^{(1)}$ the polynomials satisfying \eqref{eq:orthogonality}
with $\Gamma=\Gamma_0 \cup \Gamma_1$. The corresponding recurrence
coefficients will be accordingly denoted by $\beta_n^{(1)}$ and
$(\gamma_{n}^{(1)})^2$. Hence, we have
\begin{equation}\label{eq:orthogonality2}
\int_{\Gamma_0 \cup \Gamma_1}p_{n}^{(1)}(x)x^k e^{-x^3}d x =0,
\qquad k=0,1,\ldots,n-1,
\end{equation}
and
\begin{equation} \label{recurrence 1}
x p_n^{(1)}(x) = p_{n+1}^{(1)}(x) + \beta_n^{(1)} p_n^{(1)}(x) +
(\gamma_n^{(1)})^2 p_{n-1}^{(1)}(x).
\end{equation}
 From \eqref{recu:monic}, it is readily seen that
\begin{equation}\label{eq:beta01}
\beta_0^{(1)}=\frac{\int_{\Gamma_0 \cup \Gamma_1} x e^{-x^3}
dx}{\int_{\Gamma_0 \cup \Gamma_1} e^{-x^3}
dx}=\frac{\Gamma(2/3)}{\Gamma(1/3)}e^{\pi i/3}.
\end{equation}
Thus, one can determine $(\beta_n^{(1)}, (\gamma_{n}^{(1)})^2)$
recursively from the string equations \eqref{eq:string
1}--\eqref{eq:string 2} with initial condition $\gamma_0^{(1)}=0$
and $\beta_0^{(1)} =\frac{\Gamma(2/3)}{\Gamma(1/3)}e^{\pi i/3}$.

In a similar manner, we let $p_n^{(2)}$  be the polynomials
satisfying \eqref{eq:orthogonality} with $\Gamma=\Gamma_0 \cup
\Gamma_2$, and denote by $\beta_n^{(2)}$ and $(\gamma_{n}^{(2)})^2$
the corresponding recurrence coefficients. To this end, one has
\begin{equation}
\beta_0^{(2)}=\frac{\int_{\Gamma_0 \cup \Gamma_2} x e^{-x^3}
dx}{\int_{\Gamma_0 \cup \Gamma_2} e^{-x^3}
dx}=\overline{\beta_0^{(1)}}=\frac{\Gamma(2/3)}{\Gamma(1/3)}e^{-\pi
i/3}.
\end{equation}

For the recurrence coefficients $\beta_n^{(i)}$ and
$(\gamma_{n}^{(i)})^2$, $i=1,2$, the following proposition holds.
\begin{proposition}\label{prop:recu}
There exist two real sequences $a_n$ and $b_n$, $n\in\mathbb{N} =
\{0,1,2,3,\ldots \}$ such that
\begin{align}\label{eq:betatob gammatoa}
\beta_n^{(1)}=b_{n}e^{\pi i/3},\qquad (\gamma_n^{(1)})^2=a_ne^{-\pi
i/3},
\end{align}
and $a_n,b_n$ satisfy the coupled difference relations
\begin{align}\label{eq:an bn recur1}
a_n+a_{n+1}&=b_n^2, \\
3a_{n}(b_n+b_{n-1})&=n, \label{eq:an bn recur2}
\end{align}
with initial conditions
\begin{align}\label{eq:initial a b}
a_0=0, \qquad b_0=\frac{\Gamma(2/3)}{\Gamma(1/3)}.
\end{align}
Similarly, we have
\begin{align}\label{eq:betan2}
\beta_n^{(2)}=b_{n}e^{-\pi i/3},\qquad (\gamma_n^{(2)})^2=a_ne^{\pi
i/3},
\end{align}
with the same sequences $a_n$ and $b_n$.
\end{proposition}
 From \eqref{eq:an bn recur2}, one can easily eliminate $a_n$ in
\eqref{eq:an bn recur1} and obtain
\begin{equation}
\frac{n}{b_{n-1}+b_{n}}+\frac{n+1}{b_{n}+b_{n+1}}=3b_n^2.
\end{equation}
This difference equation belongs to $A_1^c$-type equation on the
list of discrete Painlev\'{e} equations by Grammaticos and Ramani
\cite{GR,GR2}, which has a connection with the second Painlev\'{e}
equation. It is also an alternative discrete Painlev\'e I equation
in Clarkson's list \cite[Appendix A.4]{wva-C}, see also \cite{FGR},
\cite{NSKGR}.
We give a short derivation of the string equations \eqref{eq:string 1}--\eqref{eq:string 2}
in the Appendix, where we also deal with the more general weight $e^{-x^3+tx}$.

\subsection{Multiple orthogonal polynomials with an exponential cubic weight}
Multiple orthogonal polynomials are polynomials of one variable
which are defined by orthogonality relations with respect to $r$
different measures $\mu_1,\mu_2,\allowbreak \ldots,\mu_r$, where $r
\geq 1$ is a positive integer. As a generalization of orthogonal
polynomials, multiple orthogonal polynomials originated from
Hermite-Pad\'e approximation in the context of irrationality and
transcendence proofs in number theory. They were further developed
in approximation theory, we refer to Aptekarev et al.
\cite{Apt,AptBraWVA}, Coussement and Van Assche \cite{WVAEC},
Nikishin and Sorokin \cite[Chapter 4, \S 3]{NikSor}, and Ismail
\cite[Chapter 23]{Ismail} for more information.

We take $r=2$ and for $(k,l)\in\mathbb{N}^2$, we are interested in
the monic polynomials $P_{k,l}$ of degree $k+l$ which satisfy the
orthogonality conditions
\begin{align}
\int_{\Gamma_0\cup\Gamma_1} x^i P_{k,l}(x) e^{-x^3}dx &=0,
\qquad i = 0, 1, \ldots, k-1, \label{def:ortho1}\\
\int_{\Gamma_0\cup\Gamma_2} x^i P_{k,l}(x) e^{-x^3}dx &=0, \qquad i
= 0, 1, \ldots, l-1. \label{def:ortho2}
\end{align}
We call $P_{k,l}$ the (type II) multiple orthogonal polynomial for
the exponential cubic weight. If one of $k$ and $l$ is equal to
zero, then $P_{k,l}$ reduce to the usual orthogonal polynomials with
respect to the exponential cubic weight $e^{-x^3}$, i.e.,
\begin{equation}\label{eq:kl=0}
P_{k,0}(x)=p_{k}^{(1)}(x), \qquad P_{0,k}(x)=p_{k}^{(2)}(x),
\end{equation}
where $p_{k}^{(i)}$, $i=1,2$ are defined in Section \ref{sec:op}. It
is the aim of this paper to derive some basic properties of
$P_{k,l}$. Our main results are

\begin{theorem}[Rodrigues formula]\label{thm:Rodrigues}
Let $n,m \in\mathbb{N}= \{0,1,2,3,\ldots\}$, then
\begin{align}
e^{-x^3} P_{n,n+m}(x)&=\frac{(-1)^n}{3^n} \frac{d^n}{dx^n}\left
(e^{-x^3}P_{0,m}(x) \right), \label{eq:rodri 1}\\
e^{-x^3} P_{n+m,n}(x)&=\frac{(-1)^n}{3^n} \frac{d^n}{dx^n}\left
(e^{-x^3}P_{m,0}(x) \right) \label{eq:rodri 2}.
\end{align}
where $P_{0,m}(x)$ and $P_{m,0}(x)$ are given in \eqref{eq:kl=0}.
\end{theorem}

The polynomials $P_{n,n}(x)$ were already mentioned by P\'olya and Szeg\"o in their problem book \cite[Part V, Chapter 1, Problem 59]{PS} and
P\'olya investigated their zeros in \cite[Satz IV]{Polya}. They are also a special case of polynomials introduced by Gould and Hopper \cite{GouldHopper}
and were investigated, among others, by Dominici \cite{Dom} and Paris \cite{Paris}. Their multiple orthogonality (or $d$-orthogonality, if one
only considers the diagonal polynomials) was already noted earlier, see e.g., \cite{BCBR} and references there. In this paper we are investigating the full
range of polynomials $P_{n,m}(x)$ and not only the diagonal polynomials, but we  obtain ratio asymptotics and the distribution
of the zeros for the diagonal polynomials in Section \ref{sec:zeros}. For asymptotic approximations and an asymptotic expansion of $P_{n,n}(x)$
we refer to \cite{Dom} and \cite{Paris}.

Multiple orthogonal polynomials satisfy a system of nearest-neighbor
recurrence relations \cite[Theorem 23.7]{Ismail}. For $P_{k,l}$
defined in \eqref{def:ortho1}--\eqref{def:ortho2} we can represent
the recurrence coefficients explicitly in terms of the sequences
$a_n$ and $b_n$ in Proposition \ref{prop:recu}, as stated in the
following theorem.

\begin{theorem}[the nearest-neighbor recurrence relations]\label{thm:nearest recurrence}
Let $n,m \in\mathbb{N}$, then
\begin{align}
    xP_{n,n+m}(x) = P_{n+1,n+m}(x)&+c_{n,n+m}P_{n,n+m}(x) \nonumber \\
    &+a_{n,n+m} P_{n-1,n+m}(x) + b_{n,n+m} P_{n,n+m-1}(x), \label{eq:n+1} \\
    xP_{n,n+m}(x) = P_{n,n+m+1}(x)&+d_{n,n+m}P_{n,n+m}(x) \nonumber \\
    &+ a_{n,n+m} P_{n-1,n+m}(x) + b_{n,n+m} P_{n,n+m-1}(x), \label{eq:n+m+1}
\end{align}
where
\begin{align}
c_{n,n+m}&=\left\{
            \begin{array}{ll}
              \frac{\Gamma(2/3)}{\Gamma(1/3)}e^{\pi i/3}, & \hbox{$m=0$,} \\
              -b_{m-1}e^{-\pi i/3}, & \hbox{$m>0$,}
            \end{array}
          \right.
\label{eq:c(n,n+m)}
\\
d_{n,n+m}&=b_{m}e^{-\pi i/3},
\label{eq:d(n,n+m)}\\
a_{n,n+m}&=\left\{
            \begin{array}{ll}
              -\frac{n}{3\sqrt{3}}\frac{\Gamma(1/3)}{\Gamma(2/3)}i, & \hbox{$m=0$,} \\
              -\frac{na_{m}}{m}e^{\pi i/3}, & \hbox{$m>0$,}
            \end{array}
          \right.
\label{eq:a(n,n+m)}\\
b_{n,n+m}&=\left\{
            \begin{array}{ll}
              \frac{n}{3\sqrt{3}}\frac{\Gamma(1/3)}{\Gamma(2/3)}i, & \hbox{$m=0$,} \\
              \frac{(n+m)a_m}{m}e^{\pi i/3}, & \hbox{$m>0$.}
            \end{array}
          \right.
\label{eq:b(n,n+m)}
\end{align}
Similarly,
\begin{align}
    xP_{n+m,n}(x) = P_{n+m+1,n}(x)&+c_{n+m,n}P_{n+m,n}(x) \nonumber \\
    &+a_{n+m,n} P_{n+m-1,n}(x) + b_{n+m,n} P_{n+m,n-1}(x),  \label{eq:n+1 2}\\
    xP_{n+m,n}(x) = P_{n+m,n+1}(x)&+d_{n+m,n}P_{n+m,n}(x) \nonumber \\
    &+ a_{n+m,n} P_{n+m-1,n}(x) + b_{n+m,n} P_{n+m,n-1}(x), \label{eq:n+m+1 2}
\end{align}
where
\begin{align}
c_{n+m,n}&=b_{m}e^{\pi i/3},
\\
d_{n+m,n}&=\left\{
            \begin{array}{ll}
              \frac{\Gamma(2/3)}{\Gamma(1/3)}e^{-\pi i/3}, & \hbox{$m=0$,} \\
              -b_{m-1}e^{\pi i/3}, & \hbox{$m>0$,}
            \end{array}
          \right.
\\
a_{n+m,n}&=\left\{
            \begin{array}{ll}
              -\frac{n}{3\sqrt{3}}\frac{\Gamma(1/3)}{\Gamma(2/3)}i, & \hbox{$m=0$,} \\
              \frac{(n+m)a_m}{m}e^{-\pi i/3}, & \hbox{$m>0$,}
            \end{array}
          \right.
\\
b_{n+m,n}&=\left\{
            \begin{array}{ll}
              \frac{n}{3\sqrt{3}}\frac{\Gamma(1/3)}{\Gamma(2/3)}i, & \hbox{$m=0$,} \\
              -\frac{na_m}{m}e^{-\pi i/3}, & \hbox{$m>0$.}
            \end{array}
          \right.
\end{align}
Here, $a_n$ and $b_n$ are the two real sequences generated from
\eqref{eq:an bn recur1}--\eqref{eq:initial a b}.
\end{theorem}

It is also easy to check that the recurrence coefficients derived in
Theorem \ref{thm:nearest recurrence} satisfy the partial difference
equations obtained in \cite[Theorem 3.2]{WVA}.

The rest of this paper is organized as follows. Theorems  \ref{thm:Rodrigues}
and \ref{thm:nearest recurrence}
will be proved in Section \ref{sec:proof}. The string
equation \eqref{eq:string 1} plays a particular role in the derivation of the
coefficients in the nearest-neighbor recurrence relations. We then
perform a numerical study of the coefficients $a_n$, $b_n$ in Section \ref{sec:asy of an and bn}.
The study suggests that $a_{n+1}$
and $b_n$, $n\in\mathbb{N}$ are all strictly positive, and the
limits of $a_n/n^{2/3}$ and $b_n/n^{1/3}$ exist as $n\to\infty$,
and we can identify these limits explicitly.
Section \ref{sec:zeros} deals with the zeros of $P_{k,l}$. We will give precise
location and interlacing results for the zeros of the diagonal multiple orthogonal polynomials
$P_{n,n}$ and asymptotic results for the ratio of diagonal multiple orthogonal polynomials. The latter
allows us to find the asymptotic distribution of the scaled zeros for these diagonal multiple orthogonal
polynomials. The zeros of $P_{k,l}$, with $k\neq l$, have a more interesting structure, which depends on the
limit of the ratio $k/l$. We investigate these zeros numerically and end this paper with some conclusions
and outlook.

\section{Proofs}\label{sec:proof}

\subsection{Proof of Proposition \ref{prop:recu}}
This proposition can be proved by induction on the index $n$. When
$n=0$, the relation \eqref{eq:betatob gammatoa} is obvious, which
also gives the initial conditions \eqref{eq:initial a b}. Suppose we
have
\begin{align}
\beta_k^{(1)}=b_{k}e^{\pi i/3},\qquad (\gamma_k^{(1)})^2=a_ke^{-\pi
i/3},
\end{align}
and $(a_k,b_k)\in\mathbb{R}^2$ for $k\leq n$. From \eqref{eq:string
1}, it follows that
\begin{align}
(\gamma_{n+1}^{(1)})^2&=-((\gamma_n^{(1)})^2+(\beta_n^{(1)})^2)
\nonumber
\\
&=-a_ne^{-\pi i/3}-b_n^{2}e^{2\pi i/3}=(b_n^2-a_n)e^{-\pi i/3},
\end{align}
thus,
\begin{equation}\label{eq:an recu}
a_{n+1}=b_n^2-a_n\in\mathbb{R}.
\end{equation}
On the other hand, the equation \eqref{eq:string 2} implies that
\begin{align}
\beta_{n+1}^{(1)}=\frac{n+1}{3(\gamma_{n+1}^{(1)})^2}-\beta_n^{(1)}
=\left(\frac{n+1}{3a_{n+1}}-b_n\right)e^{\pi i/3},
\end{align}
thus
\begin{equation}\label{eq:bn recu}
b_{n+1}=\frac{n+1}{3a_{n+1}}-b_n\in\mathbb{R}.
\end{equation}
The coupled difference equations \eqref{eq:an bn
recur1}--\eqref{eq:an bn recur2} are immediate from \eqref{eq:an
recu} and \eqref{eq:bn recu}.

The claim for $\beta_n^{(2)}$ and $(\gamma_n^{(2)})^2$ can be proved
similarly, we omit the details here.


\subsection{Proof of Theorem \ref{thm:Rodrigues}}
We shall only prove \eqref{eq:rodri 1} since the proof of
\eqref{eq:rodri 2} is similar.

We first show that $P_{n,n+m}$ defined in \eqref{eq:rodri 1} is a
monic polynomial of degree $2n+m$. Observe that
\begin{equation*}
\frac{(-1)^n}{3^n}\frac{d^n}{dx^n}\left(e^{-x^3}P_{0,m}\right)=
\frac{(-1)^n}{3^n}\left(\frac{d^{n-1}}{dx^{n-1}}\left(e^{-x^3}P_{0,m}\right)\right)'
=-\frac{1}{3}(e^{-x^3}P_{n-1,n+m-1}(x))',
\end{equation*}
we then obtain from \eqref{eq:rodri 1} the following
difference-differential equation for $P_{n,n+m}$:
\begin{equation}\label{eq:difference-diff eq}
P_{n,n+m}(x)=x^2P_{n-1,n+m-1}(x)-\frac{1}{3}P_{n-1,n+m-1}'(x).
\end{equation}
We can now use induction on $n$. Clearly $P_{0,m}=p_m^{(2)}$ is a
monic polynomial of degree $m$. Suppose that $P_{n-1,m+n-1}$ is a
monic polynomial of degree $2n+m-2$, then  \eqref{eq:difference-diff
eq} implies that $P_{n,n+m}$ is a monic polynomial of degree $2n+m$.

Next, we show that $P_{n,n+m}$ satisfies the orthogonality
conditions \eqref{def:ortho1}--\eqref{def:ortho2}. With $\Gamma_0$
defined in \eqref{def:contour}, it follows from \eqref{eq:rodri 1}
and integration by parts $k$ times that
\begin{align}\label{eq:int on Gamma0}
\int_{\Gamma_0} x^k P_{n,n+m}(x) e^{-x^3}\,dx
&=\frac{(-1)^n}{3^n}\int_{\Gamma_0} x^k
\frac{d^n}{dx^n}\left(e^{-x^3}P_{0,m}(x) \right)\,dx \nonumber \\
&=-\frac{(-1)^{n+k}k!}{3^n}\frac{d^{n-k-1}}{dx^{n-k-1}}\left(e^{-x^3}P_{0,m}(x)\right)\Big|_{x=0}
\nonumber \\
&=\frac{k!}{3^{k+1}}P_{n-k-1,n+m-k-1}(0),
\end{align}
for $k=0,1,\ldots,n-1$. Similarly, it is easily seen that
\begin{align}\label{eq:int on Gamma1}
\int_{\Gamma_1} x^k P_{n,n+m}(x) e^{-x^3}\,dx=\int_{\Gamma_2} x^k
P_{n,n+m}(x) e^{-x^3}\,dx =-\frac{k!}{3^{k+1}}P_{n-k-1,n+m-k-1}(0).
\end{align}
Combining \eqref{eq:int on Gamma0} and \eqref{eq:int on Gamma1}
gives
\begin{align*}
\int_{\Gamma_0\cup\Gamma_1} x^k P_{n,n+m}(x) e^{-x^3}\, dx &=0,
\qquad k = 0, 1, \ldots, n-1, \\
\int_{\Gamma_0\cup\Gamma_2} x^k P_{n,n+m}(x) e^{-x^3}\, dx &=0,
\qquad k = 0, 1, \ldots, n-1.
\end{align*}
We still need $m$ more orthogonality conditions to complete
\eqref{def:ortho2}, but these follow from
\begin{align*}
\int_{\Gamma_0\cup \Gamma_2} x^{n+k} P_{n,n+m}(x) e^{-x^3}\,dx
&=\frac{(-1)^n}{3^n}\int_{\Gamma_0\cup \Gamma_2} x^{n+k}
\frac{d^n}{dx^n}\left(e^{-x^3}P_{0,m}(x) \right)\,dx  \\
&=\frac{(n+k)!}{k!3^n}\int_{\Gamma_0 \cup \Gamma_2} x^{k}
e^{-x^3}P_{0,m}(x) \, dx = 0
\end{align*}
for $k=0,1,\ldots,m-1$, where we used the fact that
$P_{0,m}=p^{(2)}_m$ is the orthogonal polynomial for the cubic
exponential weight on $\Gamma_0\cup \Gamma_2$.

\subsection{Proof of Theorem \ref{thm:nearest recurrence}}
We will present the proof of \eqref{eq:n+1}--\eqref{eq:b(n,n+m)},
the remaining part of the theorem can be proved in a similar manner.

Let us denote the coefficients of $x^{k+l-1}$ and $x^{k+l-2}$ in
$P_{k,l}$ by $\delta_{k,l}$ and $\varepsilon_{k,l}$, respectively,
i.e.,
\begin{equation}\label{eq:deltanm}
P_{k,l}(x)=x^{k+l}+\delta_{k,l}x^{k+l-1}+\varepsilon_{k,l}x^{k+l-2}+\cdots.
\end{equation}
Substituting the above formula into \eqref{eq:difference-diff eq}
and comparing the coefficients of $x^{2n+m-1}$ and $x^{2n+m-2}$ on
both sides leads to
\begin{equation*}
\delta_{n,n+m}=\delta_{n-1,n+m-1},\qquad
\varepsilon_{n,n+m}=\varepsilon_{n-1,n+m-1},
\end{equation*}
thus,
\begin{equation}\label{eq:rel delta}
\delta_{n,n+m}=\delta_{0,m},\qquad
\varepsilon_{n,n+m}=\varepsilon_{0,m},
\end{equation}
for $m\in\mathbb{N}$. Similarly, we have
\begin{equation}\label{eq:differen-diff eq2}
P_{n+m,n}(x)=x^2P_{n+m-1,n-1}(x)-\frac{1}{3}P_{n+m-1,n-1}'(x),
\end{equation}
which implies
\begin{equation}\label{eq:rel delta2}
\delta_{n+m,n}=\delta_{m,0},\qquad
\varepsilon_{n+m,n}=\varepsilon_{m,0}.
\end{equation}
If we insert \eqref{eq:deltanm} into
\eqref{eq:n+1}--\eqref{eq:n+m+1}, then the coefficients of second
leading term $x^{2n+m}$ give
\begin{align}\label{eq:c(n,n+m) mid}
c_{n,n+m}&=\delta_{n,n+m}-\delta_{n+1,n+m}=\left\{
                                            \begin{array}{ll}
                                              \delta_{0,0}-\delta_{1,0}, & \hbox{$m=0$,} \\
                                              \delta_{0,m}-\delta_{0,m-1}, & \hbox{$m>1$,}
                                            \end{array}
                                          \right.
\\
d_{n,n+m}&=\delta_{n,n+m}-\delta_{n,n+m+1}=\delta_{0,m}-\delta_{0,m+1},
\label{eq:d(n,n+m) mid}
\end{align}
where we have also made use of the first equalities in \eqref{eq:rel
delta} and \eqref{eq:rel delta2}. On account of the facts that
\begin{align}
x P_{m,0}(x) &= P_{m+1,0}(x) + \beta_m^{(1)} P_{m,0}(x) +
(\gamma_m^{(1)})^2 P_{m-1,0}(x),
\\
x P_{0,m}(x) &= P_{0,m+1}(x) + \beta_m^{(2)} P_{0,m}(x) +
(\gamma_m^{(2)})^2 P_{0,m-1}(x), \label{eq:recu 2}
\end{align}
(see \eqref{eq:kl=0} and \eqref{recurrence 1}), it is immediate that
\begin{eqnarray}
\delta_{m,0}&=&\delta_{m+1,0}+\beta_{m}^{(1)}=\delta_{m+1,0}+b_{m}e^{\pi
i/3}, \label{eq:diff delta1} \\
\delta_{0,m}&=&\delta_{0,m+1}+\beta_{m}^{(2)}=\delta_{0,m+1}+b_{m}e^{-\pi
i/3},  \label{eq:diff delta2}
\end{eqnarray}
in view of \eqref{eq:betatob gammatoa} and \eqref{eq:betan2}. The
values for $c_{n,n+m}$, $d_{n,n+m}$ in
\eqref{eq:c(n,n+m)}--\eqref{eq:d(n,n+m)} then follow from combining
\eqref{eq:c(n,n+m) mid}, \eqref{eq:d(n,n+m) mid} and \eqref{eq:diff
delta1}--\eqref{eq:diff delta2}.

We now establish the equalities
\eqref{eq:a(n,n+m)}--\eqref{eq:b(n,n+m)} for $a_{n,n+m}$ and
$b_{n,n+m}$. Multiplying both sides of \eqref{eq:n+1} by
$x^{n+m-1}e^{-x^3}$ and integrating the equality over
$\Gamma_0\cup\Gamma_2$, the orthogonality condition
\eqref{def:ortho2} implies
\begin{equation}\label{eq:int of b(n,n+m)}
b_{n,n+m}=\frac{\ds\int_{\Gamma_0\cup\Gamma_2}x^{n+m}P_{n,n+m}(x)e^{-x^3}dx}
{\ds \int_{\Gamma_0\cup\Gamma_2}x^{n+m-1}P_{n,n+m-1}(x)e^{-x^3}dx}.
\end{equation}
By \eqref{eq:rodri 1}, \eqref{eq:rodri 2} and integrating by parts,
we find that
\begin{align}
\int_{\Gamma_0\cup\Gamma_2}x^{n+m}P_{n,n+m}(x)e^{-x^3}dx
&=\frac{(-1)^n}{3^n}\int_{\Gamma_0\cup\Gamma_2}x^{n+m}\frac{d^n}{dx^n}\left
(e^{-x^3}P_{0,m}(x) \right)dx \nonumber
\\
&=\frac{(n+m)!}{3^n m!}\int_{\Gamma_0\cup\Gamma_2}x^{m} P_{0,m}(x)
e^{-x^3}dx \nonumber
\\
&=\frac{(n+m)!}{3^n m!}\int_{\Gamma_0\cup\Gamma_2}P_{0,m}^2(x)
e^{-x^3}dx,
\end{align}
and
\begin{align}
\int_{\Gamma_0\cup\Gamma_2}x^{n-1}P_{n,n-1}(x)e^{-x^3}dx
&=\frac{(-1)^{n-1}}{3^{n-1}}\int_{\Gamma_0\cup\Gamma_2}x^{n-1}\frac{d^{n-1}}{dx^{n-1}}\left
(e^{-x^3}P_{1,0}(x) \right)dx \nonumber
\\
&=\frac{(n-1)!}{3^{n-1}}\int_{\Gamma_0\cup\Gamma_2} P_{1,0}(x)
e^{-x^3}dx.
\end{align}
Hence, we can simplify \eqref{eq:int of b(n,n+m)} as
\begin{equation}\label{eq:b(n,n+m) simp}
b_{n,n+m}=\begin{cases}
               \ds  \frac{n}{3}\frac{\int_{\Gamma_0\cup\Gamma_2}
              e^{-x^3}dx}{ \int_{\Gamma_0\cup\Gamma_2} P_{1,0}(x)
              e^{-x^3}dx}, & \hbox{$m=0$,} \\[20pt]
              \ds \frac{n+m}{m}\frac{ \int_{\Gamma_0\cup\Gamma_2}P_{0,m}^2(x)
e^{-x^3}dx}{ \int_{\Gamma_0\cup\Gamma_2}P_{0,m-1}^2(x) e^{-x^3}dx},
& \hbox{$m>0$.}
            \end{cases}
\end{equation}
Note that
\begin{equation}\label{int:P0m}
\int_{\Gamma_0\cup\Gamma_2}P_{0,m}^2(x)
e^{-x^3}dx=(\gamma_1^{(2)}\gamma_2^{(2)}\cdots
\gamma_m^{(2)})^2\int_{\Gamma_0\cup\Gamma_2} e^{-x^3}dx,\quad m>0,
\end{equation}
and straightforward calculations give us
\begin{equation}
\int_{\Gamma_0\cup\Gamma_2} e^{-x^3}dx
=\frac{\Gamma(1/3)}{3}(1-\omega^2),
\end{equation}
\begin{align}\label{int:P10}
\int_{\Gamma_0\cup\Gamma_2} P_{1,0}(x) e^{-x^3}dx
&=\int_{\Gamma_0\cup\Gamma_2} (x-\beta_0^{(1)}) e^{-x^3}dx
\nonumber \\
&=\frac{\Gamma(2/3)}{3}(1-\omega)-\frac{\Gamma(1/3)\beta_0^{(1)}}{3}(1-\omega^2)
\nonumber \\
&=\frac{\Gamma(2/3)}{3}(1-\omega)(1-(1+\omega)e^{\pi i/3}).
\end{align}
See \eqref{eq:beta01} for the value of $\beta_0^{(1)}$. Inserting
\eqref{int:P0m}--\eqref{int:P10} into \eqref{eq:b(n,n+m) simp}, we
arrive at
\begin{equation}\label{eq:b(n,n+m)2}
b_{n,n+m}=\begin{cases}
              \frac{n}{3\sqrt{3}}\frac{\Gamma(1/3)}{\Gamma(2/3)}i, & \hbox{$m=0$,} \\[10pt]
             \frac{n+m}{m}(\gamma_m^{(2)})^2, & \hbox{$m>0$,}
            \end{cases}
\end{equation}
which is \eqref{eq:b(n,n+m)} by \eqref{eq:betan2}.

We can also represent $a_{n,n+m}$ as a ratio of two integrals.
Indeed, by performing similar strategies above, it is easily seen
that
\begin{equation}\label{eq:int of a(n,n+m)}
a_{n,n+m}=\frac{\ds\int_{\Gamma_0\cup\Gamma_1}x^{n}P_{n,n+m}(x)e^{-x^3}dx}
{\ds\int_{\Gamma_0\cup\Gamma_1}x^{n-1}P_{n-1,n+m}(x)e^{-x^3}dx}
=\frac{n}{3}\frac{\ds\int_{\Gamma_0\cup\Gamma_1}P_{0,m}(x)e^{-x^3}dx}
{\ds\int_{\Gamma_0\cup\Gamma_1}P_{0,m+1}(x)e^{-x^3}dx}.
\end{equation}
Unfortunately, this representation is not suitable for direct
calculation except for $m=0$, which gives
\begin{align}\label{eq:int of a(n,n)}
a_{n,n}&= \frac{n}{3}\frac{\ds
\int_{\Gamma_0\cup\Gamma_1}P_{0,0}(x)e^{-x^3}dx}
{\ds\int_{\Gamma_0\cup\Gamma_1}P_{0,1}(x)e^{-x^3}dx} \nonumber
\\
&=\frac{n}{3}\frac{\ds\int_{\Gamma_0\cup\Gamma_1}e^{-x^3}dx}
{\ds\int_{\Gamma_0\cup\Gamma_1}(x-\beta_0^{(2)})e^{-x^3}dx} \nonumber \\
&=\overline{b_{n,n}}=-\frac{n}{3\sqrt{3}}\frac{\Gamma(1/3)}{\Gamma(2/3)}i.
\end{align}
For $m>0$ the integrals of $P_{0,m}$ over $\Gamma_0\cup\Gamma_{1}$
are involving polynomials orthogonal on the contour
$\Gamma_0\cup\Gamma_{2}$, hence it is then difficult to deal with
them. So we proceed in another way and we calculate the sum
$a_{n,n+m}+b_{n,n+m}$. Recall the notation $\delta_{k,l}$ and
$\varepsilon_{k,l}$ in \eqref{eq:deltanm}. By comparing the
coefficient of $x^{2n+m-1}$ on both sides of \eqref{eq:n+1}, it
follows from \eqref{eq:rel delta} that
\begin{equation}\label{eq:sum a and b}
a_{n,n+m}+b_{n,n+m}=\varepsilon_{n,n+m}-\varepsilon_{n+1,n+m}-c_{n,n+m}\delta_{n,n+m}
=\varepsilon_{0,m}-\varepsilon_{0,m-1}-c_{n,n+m}\delta_{0,m}
\end{equation}
for $m>0$. From \eqref{eq:recu 2}, we have
\begin{equation}
\varepsilon_{0,m}=\varepsilon_{0,m+1}+\beta_{m}^{(2)}\delta_{0,m}+(\gamma_m^{(2)})^2.
\end{equation}
This, together with \eqref{eq:betan2}, \eqref{eq:c(n,n+m)} and \eqref{eq:sum a and b},
implies
\begin{align}\label{eq:sum a and b 2}
a_{n,n+m}+b_{n,n+m}
&=\varepsilon_{0,m}-\varepsilon_{0,m-1}-c_{n,n+m}\delta_{0,m} \nonumber \\
&=-\beta_{m-1}^{(2)}\delta_{0,m-1}-(\gamma_{m-1}^{(2)})^2+\beta_{m-1}^{(2)}\delta_{0,m}
\nonumber \\
&=\beta_{m-1}^{(2)}(\delta_{0,m}-\delta_{0,m-1})-(\gamma_{m-1}^{(2)})^2
 \nonumber \\
&=-(\beta_{m-1}^{(2)})^2-(\gamma_{m-1}^{(2)})^2
\nonumber \\
&=(\gamma_{m}^{(2)})^2,\qquad m>0,
\end{align}
where we have made use of \eqref{eq:diff delta2} in the fourth
equality and the string equation \eqref{eq:string 1} in the last
step. A combination of \eqref{eq:b(n,n+m)2}, \eqref{eq:sum a and b
2} and \eqref{eq:int of a(n,n)} finally gives
\begin{equation}
a_{n,n+m}=\begin{cases}
             -\frac{n}{3\sqrt{3}}\frac{\Gamma(1/3)}{\Gamma(2/3)}i, & \hbox{$m=0$,} \\[10pt]
              -\frac{n}{m}(\gamma_m^{(2)})^2, & \hbox{$m>0$,}
            \end{cases}
\end{equation}
which is \eqref{eq:a(n,n+m)}, on account of \eqref{eq:betan2}.


\section{Asymptotics of $a_n$ and $b_n$} \label{sec:asy of an and bn}

From Theorem \ref{thm:nearest recurrence}, it is clear that the
coefficients in the nearest-neighbor recurrence relations are
determined by $a_n$ and $b_n$ generated from \eqref{eq:an bn
recur1}--\eqref{eq:initial a b}. It is then interesting to study
their large $n$ behavior. In Figure \ref{fig:1} we have plotted the
values of $a_n/n^{2/3}$ and $b_n/n^{1/3}$ for $n$ from $0$ to $70$,
from which we see that $a_{n+1}$ and $b_n$ are all strictly positive
for $n\in\mathbb{N}$. We actually have the following conjecture
concerning this observation.

\begin{figure}[ht]
\begin{center}
 \resizebox*{7cm}{!}{\includegraphics{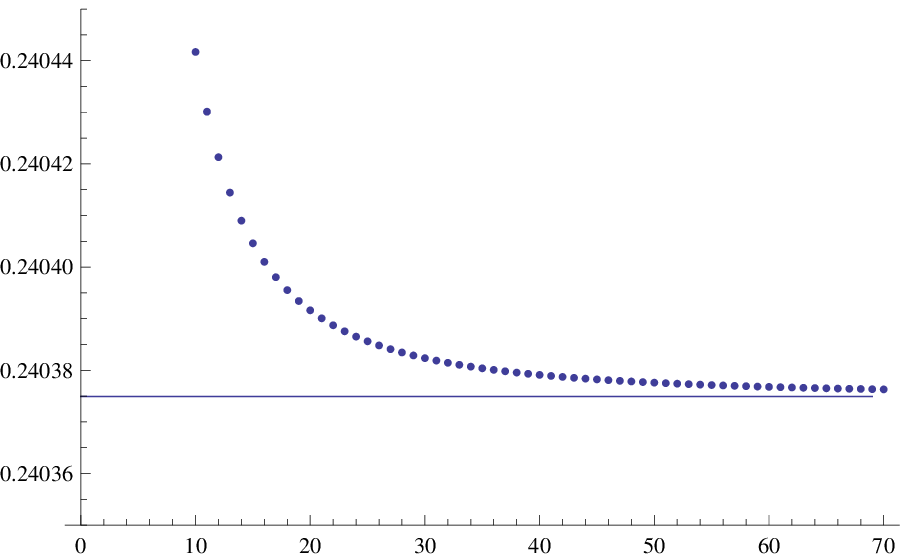}}
 \hspace{2mm} \resizebox*{7cm}{!}{\includegraphics{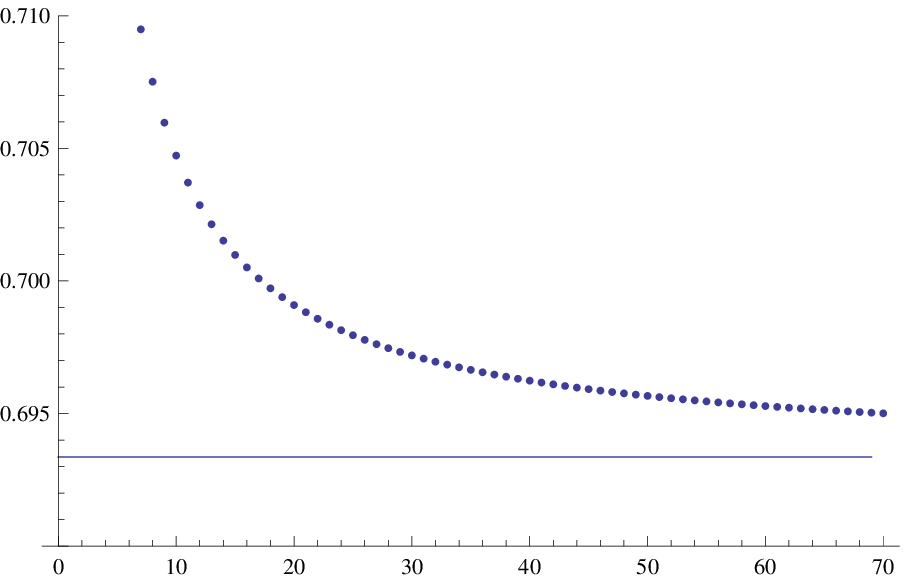}}
 \caption{The values of $a_n/n^{2/3}$ (left) and $b_n/n^{1/3}$
(right) for $n$ from $0$ to $70$.} \label{fig:1}
\end{center}
\end{figure}

\begin{conjecture}
There is a unique positive solution of the recurrence relations
\eqref{eq:an bn recur1}--\eqref{eq:an bn recur2} with $a_0=0$ and
$a_{n+1}>0, b_n>0$ for $n\in\mathbb{N}$. This solution corresponds
to the initial condition $b_0=\Gamma(2/3)/\Gamma(1/3)$.
\end{conjecture}

The numerical study further suggests that the limits of
$a_n/n^{2/3}$ and $b_n/n^{1/3}$ exist as $n\to\infty$, which we can
identify in the proposition below.
\begin{proposition}  \label{prop:anbn}
Every positive solution of \eqref{eq:an bn recur1}--\eqref{eq:an bn
recur2} has the property that
$$
\lim_{n\to\infty}a_n/n^{2/3}=\frac{1}{2\cdot 3^{2/3}},\qquad
\lim_{n\to\infty}b_n/n^{1/3}=\frac{1}{ 3^{1/3}}.
$$
\end{proposition}

\begin{proof}
This can be proved by an argument which was already used by Freud in
\cite[\S 3]{FR}. First we show that $(a_n/n^{2/3})_{n \in
\mathbb{N}}$ is a bounded sequence. From \eqref{eq:an bn recur1} and
the positivity of $a_{n+1}$ we find that $a_n \leq b_n^2$. From
\eqref{eq:an bn recur2} and the positivity of $b_{n-1}$ we find
$3a_nb_n \leq n$ and thus also $9a_n^2b_n^2 \leq n^2$. Together this
gives $9a_n^3 \leq n^2$, so that $0 \leq a_n/n^{2/3} \leq
1/9^{1/3}$.

Let $a= \liminf_{n \to \infty} a_n/n^{2/3}$ and $A = \limsup_{n \to
\infty} a_n/n^{2/3}$, then $0 \leq a \leq A < \infty$. From
\eqref{eq:an bn recur1} and the positivity of $b_n$ we find $b_n =
\sqrt{a_n+a_{n+1}}$. Insert this in \eqref{eq:an bn recur2} to find
\begin{equation}  \label{eq:an recur}
  3a_n \left( \sqrt{a_n+a_{n+1}} + \sqrt{a_n+a_{n-1}} \right) = n.
\end{equation}
Let $n \to \infty$ in \eqref{eq:an recur} through a subsequence for
which $a_n/n^{2/3} \to a$, then one finds $1 \leq 6a \sqrt{a+A}$. If
$n \to \infty$ through a subsequence for which $a_n/n^{2/3} \to A$,
then $6A \sqrt{a+A} \leq 1$. Together this gives $6A \sqrt{a+A} \leq
6a \sqrt{a+A}$. If $a+A=0$ then one automatically has $a=A=0$ so
that the limit exists (further on we will see that $a\neq 0$ so that
this case does not happen).  If $a+A >0$ then one finds $A \leq a$,
and together with $a \leq A$ we see that also in this case $a=A$ and
the limit exists.
 From \eqref{eq:an bn recur1} we then find that $\lim_{n \to \infty} b_n/n^{1/3} = \sqrt{2a}$. If we use that information in  \eqref{eq:an bn recur2}, then
 $6a\sqrt{2a} = 1$, so that $a = (1/6)^{2/3} = 1/(2\cdot 3^{2/3})$. The limit for $b_n/n^{1/3}$ follows immediately from this.
\end{proof}


\begin{figure}[ht]
\begin{center}
\includegraphics{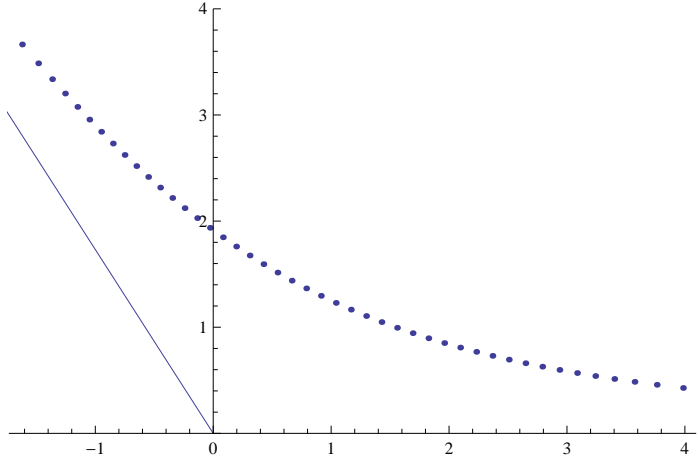}
\caption{Zeros of $P_{45,0}$ (after scaling)} \label{fig:2}
\end{center}
\end{figure}

\section{Zeros} \label{sec:zeros}

The formulas in Theorems \ref{thm:Rodrigues} and \ref{thm:nearest
recurrence} can be used to generate the multiple orthogonal
polynomials $P_{k,l}$ defined in \eqref{def:ortho1} and
\eqref{def:ortho2}. We investigate the distribution of their zeros
numerically. If one of $k$ and $l$ is zero, the polynomials are
orthogonal for the exponential cubic weight on the curve
($\Gamma_0\cup\Gamma_1$ or $\Gamma_0\cup\Gamma_2$) in the complex
plane. The zeros of $P_{45,0}(45^{1/3}x)$ are plotted in Figure \ref{fig:2}.
It is known that, in this case, the zeros of the polynomials, after proper scaling,
will accumulate on an analytic contour in the complex plane that
possesses the so-called $S-$property; cf. \cite{GoR,Stahl}. The zero
distribution was investigated earlier by Dea\~{n}o, Huybrechs and
Kuijlaars \cite{DHK}, who in fact used the weight $e^{ix^3}$.
However, a simple rotation $x = ye^{\pi i/6}$ is enough to transform
their results to the exponential cubic $e^{-y^3}$ which we are
using.

\begin{figure}[t]
\begin{center}
\includegraphics{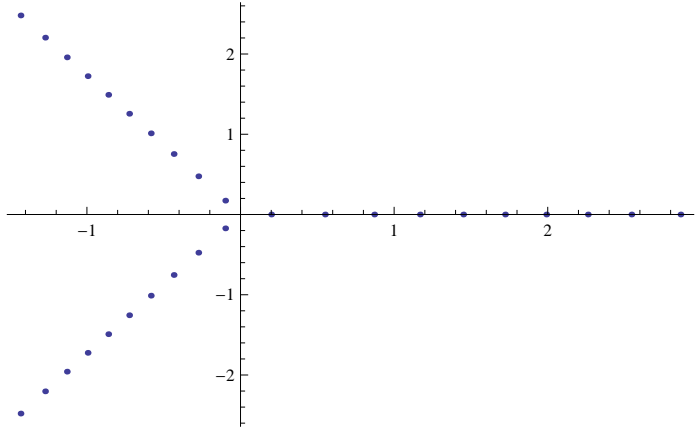}
\caption{Zeros of $P_{15,15}$ (after scaling)} \label{fig:3}
\end{center}
\end{figure}

Suppose that $k=l=n$. It follows from Theorem \ref{thm:Rodrigues} that
\begin{align}
P_{n,n}(x)=\frac{(-1)^n}{3^n}e^{x^3}\frac{d^n}{dx^n}\left (e^{-x^3}
\right).
\end{align}
We can describe the asymptotic distribution of the zeros of the diagonal polynomials $P_{n,n}$ in more detail.
The main reason is that the zeros of $P_{n,n}$ are all located on the three rays $\Gamma_0 \cup \Gamma_1 \cup \Gamma_2$,
which simplifies matters considerably (see Figure \ref{fig:3}). We have the following result for the diagonal polynomials.
Observe that this result is the solution of Problem 59 \cite[Part V, Chapter 1]{PS} for the polynomial $R_n$ and $q=3$.
\begin{proposition}   \label{prop:sym}
The polynomials $P_{n,n}(x)$ satisfy the symmetry property $P_{n,n}(\omega x) = \omega^{2n}P_{n,n}(x)$, where $\omega = e^{2\pi i/3}$ is the primitive third root of unity. In particular
\begin{equation}  \label{eq:Pnn}
               P_{n,n}(x) = \begin{cases}
                    \sum_{j=0}^{2n/3} a_j x^{3j}, & n \equiv 0 \bmod 3, \\
                     x^2 \sum_{j=0}^{2(n-1)/3} b_j x^{3j}, & n \equiv 1 \bmod 3, \\
                     x \sum_{j=0}^{2(n-2)/3+1} c_j x^{3j}, & n \equiv 2 \bmod 3,
                   \end{cases}
\end{equation}
where $(a_j)_j,(b_j)_j,(c_j)_j$ are real sequences depending on $n$. Furthermore the number of strictly positive real zeros of $P_{n,n}$ is
\[   \begin{cases}
       \frac{2n}{3}, & \textrm{if } n \equiv 0 \bmod 3, \\
       \frac{2(n-1)}{3}, & \textrm{if } n \equiv 1 \bmod 3, \\
       \frac{2(n-2)}{3}+1, & \textrm{if } n \equiv 2 \bmod 3,
     \end{cases}  \]
and $P_{n,n}(x)$ has a zero of multiplicity one at $x=0$ when $n \equiv 2 \bmod3$ and a zero of multiplicity two at $x=0$
when $n \equiv 1 \bmod 3$.
\end{proposition}

\begin{proof}
We use induction on $n$. The symmetry property follows easily from the Rodrigues formula, so we only need to prove the result about the
positive real zeros. Observe that
\[   P_{0,0}(x)=1, \quad P_{1,1}(x) = x^2, \quad P_{2,2}(x) = x(x^3-2/3), \]
so that the result is true for $n=0,1,2$. Suppose that the result is true for $n-1$ and let $x_1 > x_2 > \cdots > x_k > 0$ be the
positive real zeros of $P_{n-1,n-1}$. Clearly the sign of $P_{n-1,n-1}'(x_j)$ is $(-1)^{j+1}$ for $1 \leq j \leq k$, hence from
\begin{equation} \label{eq:Pnnrec}
 P_{n,n}(x) = x^2 P_{n-1,n-1}(x) - \frac13 P_{n-1,n-1}'(x)
\end{equation}
we find that the sign of $P_{n,n}(x_j)$ is $(-1)^j$, hence $P_{n,n}$ changes sign $k$ times and Rolle's theorem guarantees that there are
at least $k$ zeros $y_1 > y_2 > \cdots > y_k$ with $x_j < y_j < x_{j-1}$, where $x_0 = +\infty$.
\begin{itemize}
  \item If $n \equiv 0 \bmod 3$ then $n-1 \equiv 2 \bmod 3$ and the induction hypothesis says that $k=2(n-3)/3+1$ is odd and $P_{n-1,n-1}(x)$ has a zero
  of multiplicity one at $x=0$. The sign of $P_{n-1,n-1}'(0)$ is $(-1)^{k}=-1$ so that the sign of $P_{n,n}(0)$ is $(-1)^{k+1}=1$, hence there is also a zero $y_{k+1}$ of $P_{n,n}$ between $0$ and $x_k$, giving a total of $k+1= 2n/3$ positive real zeros. The $\omega$-symmetry gives another $2n/3$ zeros on $\Gamma_1$ and $2n/3$ zeros on $\Gamma_2$, which is a total of $2n$ zeros. Hence there are no other zeros of $P_{n,n}$.
  \item If $n \equiv 1 \bmod 3$ then $n-1 \equiv 0 \bmod 3$ and the induction hypothesis gives $k=2(n-1)/3$ and $P_{n-1,n-1}(x)$ has no zero at $x=0$.
  Hence there will not be an additional zero between $0$ and $x_k$ so that there are $k=2(n-1)/3$ positive real zeros for $P_{n,n}$. There is
  double zero of $P_{n,n}(x)$ at $x=0$. The $\omega$-symmetry gives another $2(n-1)/3$ zeros on $\Gamma_1$ and $2(n-1)/3$ zeros on $\Gamma_2$, hence
  the total number of zeros is $2(n-2)+2=2n$ so that there are no other zeros.
  \item If $n \equiv 2 \bmod 3$ then $n-1 \equiv 1 \bmod 3$ and the induction hypothesis gives $k=2(n-2)/3$ is even and a double zero for $P_{n-1,n-1}(x)$ at $x=0$. Then \eqref{eq:Pnnrec} implies that $P_{n,n}(x)$ has a single zero at $0$. The polynomial $P_{n-1,n-1}(x)/x^2$ of degree $2n-4$ has $k$ positive zeros and the sign of this polynomial as $x \to 0$ is $(-1)^k=1$, so that $P_{n,n}(x)/x$ has sign $(-1)^{k+1}=-1$ as $x \to 0$. Hence $P_{n,n}(x)/x$ has a zero $y_{k+1}$ between $0$ and $x_k$, giving a total of $k+1=2(n-2)/3+1$ positive real zeros. The $\omega$-symmetry gives another $2(n-2)/3+1$ zeros on $\Gamma_1$ and another $2(n-2)/3+1$ zeros on $\Gamma_2$, hence together with the single zero at $x=0$ this gives a total of $2(n-2)+4= 2n$ zeros for $P_{n,n}$ so that there are no other zeros.
\end{itemize}
\end{proof}

Observe that the proof also shows that the zeros of $P_{n-1,n-1}$ and $P_{n,n}$ interlace in the sense that $x_1 < y_1 < \infty$, $x_j < y_j < x_{j-1}$
for $j=2,\ldots,k$, $x_1 < y_1 < \infty$ and $0 < y_{k+1} < x_k$ (the latter only when $n \equiv 0 \bmod 3$ and $n \equiv 2 \bmod 3$).

We can now prove the following results

\begin{theorem}  \label{thm:ratio}
Let $K$ be a compact set in $\mathbb{C} \setminus (\Gamma_0 \cup \Gamma_1 \cup \Gamma_2)$, then
\[  \lim_{n \to \infty} \frac{1}{n^{2/3}} \frac{P_{n,n}(n^{1/3}x)}{P_{n-1,n-1}(n^{1/3}x)} = \Phi(x), \]
holds uniformly for $x \in K$, where
\[  \Phi(x) = \frac{1}{e^{2\pi i/3} \left(\frac{-3+\sqrt{9-4x^3}}{2}\right)^{2/3} + e^{-2\pi i/3} \left( \frac{-3-\sqrt{9-4x^3}}2 \right)^{2/3} + 2x}. \]
Furthermore
\[   \lim_{n \to \infty} \frac{1}{n^{2/3}} \frac{P_{n,n}'(n^{1/3}x)}{P_{n,n}(n^{1/3}x)} = 3x^2 - 3 \Phi(x), \]
holds uniformly for $x \in K$.
\end{theorem}

\begin{proof}
Consider the ratio
\[   \frac{1}{N} \frac{\frac{d}{dx} P_{n,n}(N^{1/3} x)}{P_{n,n}(N^{1/3}x)} = \frac{1}{N^{2/3}} \frac{P_{n,n}'(N^{1/3}x)}{P_{n,n}(N^{1/3}x)}
    = \frac{1}{N} \sum_{j=1}^{2n} \frac{1}{x-x_{j,n}/N^{1/3}}, \]
where $\{x_{j,n}, 1 \leq j \leq 2n \}$ are the zeros of $P_{n,n}$ which are all on the set $\Gamma_0 \cup \Gamma_1 \cup \Gamma_2$,
then if $x \in K$ we have
\[   \left| \frac{1}{N^{2/3}} \frac{P_{n,n}'(N^{1/3}x)}{P_{n,n}(N^{1/3}x)} \right| \leq \frac{1}{N} \sum_{j=1}^{2n} \frac{1}{|x-x_{j,n}/N^{1/3}|}
\leq \frac{2n}{N\delta}, \]
where $\delta = \inf \{ |x-y|\ :\  x \in K, y \in \Gamma_0 \cup \Gamma_1 \cup \Gamma_2\} > 0$ is the minimal distance between $K$ and
$\Gamma_0 \cup \Gamma_1 \cup \Gamma_2$. If $n/N \to 1$ we then see that the family of analytic functions
\[    \frac{1}{N^{2/3}} \frac{P_{n,n}'(N^{1/3}x)}{P_{n,n}(N^{1/3}x)}, \]
is uniformly bounded on $K$. By Montel's theorem there exists a subsequence $(n_k)_{k \in \mathbb{N}}$ such that
\begin{equation}  \label{eq:F2}
  \lim_{n_k \to \infty} \frac{1}{n_k^{2/3}} \frac{P_{n_k-2,n_k-2}'(n_k^{1/3}x)}{P_{n_k-2,n_k-2}(n_k^{1/3}x)} = F(x),
\end{equation}
uniformly for $x \in K$, where $F$ is an analytic function on $K$ for which $F(x) = 2/x + \mathcal{O}(1/x^2)$ as
$x \to \infty$. This function $F$ may depend on the selected subsequence, so our aim is to prove that it is independent of the subsequence.

Now consider \eqref{eq:Pnnrec} for $P_{n_k-1,n_k-1}$, then
\[   \frac{1}{N^{2/3}} \frac{P_{n_k-1,n_k-1}(N^{1/3}x)}{P_{n_k-2,n_k-2}(N^{1/3}x)}
= x^2 - \frac{1}{3N^{2/3}} \frac{P_{n_k-2,n_k-2}'(N^{1/3}x)}{P_{n_k-2,n_k-2}(N^{1/3}x)}, \]
hence \eqref{eq:F2} implies (with $N=n_k$)
\begin{equation}  \label{eq:Phi2}
   \lim_{n_k \to \infty} \frac{1}{n_k^{2/3}} \frac{P_{n_k-1,n_k-1}(n_k^{1/3}x)}{P_{n_k-2,n_k-2}(n_k^{1/3}x)} = \Phi(x),
\end{equation}
uniformly on $K$, where $\Phi(x) = x^2 - F(x)/3$.
This uniform convergence of analytic functions implies also the uniform convergence of the derivatives, hence
\begin{multline*}
  \Phi'(x)  = \lim_{n_k \to \infty} \left( \frac{1}{n_k^{2/3}} \frac{P_{n_k-1,n_k-1}(n_k^{1/3}x)}{P_{n_k-2,n_k-2}(n_k^{1/3}x)} \right)' \\
    = \lim_{n_k \to \infty} \frac{1}{n_k^{2/3}} \frac{P_{n_k-1,n_k-1}(n_k^{1/3}x)}{P_{n_k-2,n_k-2}(n_k^{1/3}x)}
   \left(   \frac{n_k^{1/3} P_{n_k-1,n_k-1}'(n_k^{1/3}x)}{P_{n_k-1,n_k-1}(n_k^{1/3}x)} -
       \frac{n_k^{1/3} P_{n_k-2,n_k-2}'(n_k^{1/3}x)}{P_{n_k-2,n_k-2}(n_k^{1/3}x)} \right),
\end{multline*}
but this means that
\begin{equation}  \label{eq:F1}
  \lim_{n_k \to \infty} \frac{1}{n_k^{2/3}} \frac{P_{n_k-1,n_k-1}'(n_k^{1/3}x)}{P_{n_k-1,n_k-1}(n_k^{1/3}x)} = F(x),
\end{equation}
uniformly on $K$, with the same limit as in \eqref{eq:F2}. But then \eqref{eq:Pnnrec} implies that
\begin{equation}  \label{eq:Phi1}
   \lim_{n_k \to \infty} \frac{1}{n_k^{2/3}} \frac{P_{n_k,n_k}(n_k^{1/3}x)}{P_{n_k-1,n_k-1}(n_k^{1/3}x)} = \Phi(x),
\end{equation}
uniformly on $K$, with the same limit as in \eqref{eq:Phi2}. We can repeat this reasoning once more and find that
\begin{equation}  \label{eq:F0}
   \lim_{n_k \to \infty} \frac{1}{n_k^{2/3}} \frac{P_{n_k,n_k}'(n_k^{1/3}x)}{P_{n_k,n_k}(n_k^{1/3}x)} = F(x),
\end{equation}
and
\begin{equation}  \label{eq:Phi0}
   \lim_{n_k \to \infty} \frac{1}{n_k^{2/3}} \frac{P_{n_k+1,n_k+1}(n_k^{1/3}x)}{P_{n_k,n_k}(n_k^{1/3}x)} = \Phi(x),
\end{equation}
uniformly on $K$. 

We will show that the function $\Phi$ satisfies a cubic equation, from which we can determine $\Phi$ and hence also $F$ uniquely,
so that $\Phi$ and $F$ do not depend on the selected subsequence $(n_k)_{k \in \mathbb{N}}$.
Consider the nearest neighbor recurrence relations for the diagonal $n=m$
\begin{eqnarray}
  xP_{n,n}(x) &=& P_{n+1,n}(x) + c_{n,n} P_{n,n}(x) + a_{n,n} P_{n-1,n}(x) + b_{n,n} P_{n,n-1}(x) \label{eq:Pnn1} \\
  xP_{n,n}(x) &=& P_{n,n+1}(x) + d_{n,n} P_{n,n}(x) + a_{n,n} P_{n-1,n}(x) + b_{n,n} P_{n,n-1}(x). \label{eq:Pnn2}
\end{eqnarray}
Subtracting \eqref{eq:Pnn1} and \eqref{eq:Pnn2} gives
\[    P_{n+1,n}(x) - P_{n,n+1}(x) = (d_{n,n}-c_{n,n})P_{n,n}(x). \]
Use this for $n \to n-1$ to eliminate $P_{n-1,n}(x)$ in \eqref{eq:Pnn1} to find
\begin{multline*}
  xP_{n,n}(x) = P_{n+1,n}(x) + c_{n,n} P_{n,n}(x) + (a_{n,n}+b_{n,n}) P_{n,n-1}(x) \\
  +\ a_{n,n}(c_{n-1,n-1}-d_{n-1,n-1}) P_{n-1,n-1}(x).
\end{multline*}
 From Theorem \ref{thm:nearest recurrence} we have
\[  c_{n,n} = \frac{\Gamma(2/3)}{\Gamma(1/3)} e^{\pi i/3}, \quad  d_{n,n} = \frac{\Gamma(2/3)}{\Gamma(1/3)} e^{-\pi i/3}, \]
so that $c_{n-1,n-1}-d_{n-1,n-1}= i \sqrt{3}\Gamma(2/3)/\Gamma(1/3)$. Furthermore
\[  a_{n,n} = -\frac{ni}{3\sqrt{3}} \frac{\Gamma(1/3)}{\Gamma(2/3)}, \quad b_{n,n} = \frac{ni}{3\sqrt{3}} \frac{\Gamma(1/3)}{\Gamma(2/3)}, \]
so that the recurrence relation becomes
\begin{equation}  \label{eq:Pstep}
  xP_{n,n}(x) = P_{n+1,n}(x) + \frac{\Gamma(2/3)}{\Gamma(1/3)} e^{\pi i/3} P_{n,n}(x) + \frac{n}{3} P_{n-1,n-1}(x).
\end{equation}
Use this for $n^{1/3}x$ and divide by $P_{n,n}(n^{1/3}x)$, then
\[   x = \frac{1}{n^{1/3}} \frac{P_{n+1,n}(n^{1/3}x)}{P_{n,n}(n^{1/3}x)} + \frac{1}{n^{1/3}} \frac{\Gamma(2/3)}{\Gamma(1/3)} e^{\pi i/3}
    + \frac{n}{3n^{1/3}} \frac{P_{n-1,n-1}(n^{1/3}x)}{P_{n,n}(n^{1/3}x)}, \]
and by using \eqref{eq:Phi1} we find
\begin{equation}  \label{eq:Pnn+0}
  \lim_{n_k\to \infty} \frac{1}{n_k^{1/3}} \frac{P_{n_k+1,n_k}(n_k^{1/3}x)}{P_{n_k,n_k}(n_k^{1/3}x)} =
     x - \frac{1}{3\Phi(x)},
\end{equation}
uniformly on $K$. We can repeat the reasoning for $n \to n-1$ and use \eqref{eq:Phi2} to find
\begin{equation}  \label{eq:Pnn+1}
  \lim_{n_k\to \infty} \frac{1}{n_k^{1/3}} \frac{P_{n_k,n_k-1}(n_k^{1/3}x)}{P_{n_k-1,n_k-1}(n_k^{1/3}x)} =
     x - \frac{1}{3\Phi(x)},
\end{equation}
uniformly on $K$. But then the uniform convergence also holds for the derivative, and as before \eqref{eq:F1} then implies that
\[    \lim_{n_k \to \infty} \frac{1}{n_k^{2/3}} \frac{P_{n_k,n_k-1}'(n_k^{1/3}x)}{P_{n_k,n_k-1}(n_k^{1/3}x)} = F(x). \]
Use \eqref{eq:Pnnrec} for $P_{n+1,n}(n^{1/3}x)$ and divide by $P_{n,n-1}(n^{1/3}x)$, then the latter asymptotic result gives
\begin{equation}  \label{eq:Phi+1}
  \lim_{n_k \to \infty} \frac{1}{n_k^{2/3}} \frac{P_{n_k+1,n_k}(n_k^{1/3}x)}{P_{n_k,n_k-1}(n_k^{1/3}x)} = x^2 - \frac13 F(x) = \Phi(x),
\end{equation}
uniformly on $K$.

In a similar way as before, the nearest neighbor recurrence relations for $(n+1,n)$ can be transformed to
\begin{multline*}
  xP_{n+1,n}(x) = P_{n+1,n+1}(x) + d_{n+1,n} P_{n+1,n}(x) + (a_{n+1,n}+b_{n+1,n}) P_{n,n}(x) \\
  +\ b_{n+1,n}(d_{n,n-1}-c_{n,n-1}) P_{n,n-1}(x).
\end{multline*}
 From Theorem \ref{thm:nearest recurrence} we now use
\[   d_{n+1,n} = - b_0 e^{\pi i/3}, \quad c_{n+1,n} = b_1 e^{\pi i/3}, \]
so that $d_{n,n-1}-c_{n,n-1} = -(b_0+b_1)e^{\pi i/3} =-e^{\pi i/3}/(3a_1)$, where we used \eqref{eq:an bn recur2} with $n=1$.
We also have
\[   a_{n+1,n} = (n+1)a_1 e^{-\pi i/3}, \quad b_{n+1,n} = -na_1e^{-\pi i/3}, \]
so that the recurrence relation becomes
\begin{equation}  \label{eq:Pstep2}
  xP_{n+1,n}(x) = P_{n+1,n+1}(x) - \frac{\Gamma(2/3)}{\Gamma(1/3)} e^{\pi i/3} P_{n+1,n}(x) + a_1e^{-\pi i/3} P_{n,n}(x)
+ \frac{n}{3} P_{n,n-1}(x).
\end{equation}
Consider this for $n^{1/3}x$ and divide by $P_{n+1,n}(n^{1/3}x)$ then
\begin{multline*}
  x = \frac{1}{n^{1/3}} \frac{P_{n+1,n+1}(n^{1/3}x)}{P_{n+1,n}(n^{1/3}x)} - \frac{1}{n^{1/3}} \frac{\Gamma(2/3)}{\Gamma(1/3)} e^{\pi i/3}
    + \frac{a_1}{n^{1/3}} e^{-\pi i/3} \frac{P_{n,n}(n^{1/3}x)}{P_{n+1,n}(n^{1/3}x)} \\
     + \frac{n}{3n^{1/3}} \frac{P_{n,n-1}(n^{1/3}x)}{P_{n+1,n}(n^{1/3}x)}
\end{multline*}
and by using \eqref{eq:Pnn+0} and \eqref{eq:Phi+1} we find
\begin{equation}  \label{eq:Pnn+2}
  \lim_{n_k \to \infty} \frac{1}{n_k^{1/3}} \frac{P_{n_k+1,n_k+1}(n_k^{1/3}x)}{P_{n_k+1,n_k}(n_k^{1/3}x)} = x - \frac{1}{3\Phi(x)},
\end{equation}
uniformly on $K$.

Now use the relation
\[   \frac{1}{n^{2/3}} \frac{P_{n+1,n+1}(n^{1/3}x)}{P_{n,n}(n^{1/3}x)} =
  \frac{1}{n^{1/3}} \frac{P_{n+1,n+1}(n^{1/3}x)}{P_{n+1,n}(n^{1/3}x)} \frac{1}{n^{1/3}} \frac{P_{n+1,n}(n^{1/3}x)}{P_{n,n}(n^{1/3}x)}
\]
and let $n \to \infty$ through the subsequence $(n_k)_{k \in \mathbb{N}}$, then \eqref{eq:Phi0}, \eqref{eq:Pnn+0} and \eqref{eq:Pnn+2} show that
\begin{equation}  \label{eq:cubic}
   \Phi(x) = \left( x- \frac{1}{3\Phi(x)} \right)^2.
\end{equation}
The cubic equation \eqref{eq:cubic} has one solution $\Phi_1$ which behaves for $x \to \infty$ as
\[  \Phi_1(x) = x^2 + \mathcal{O}(1/x), \qquad x \to \infty.  \]
There are two other solutions $\Phi_{2,3}$ which behave as $1/(3x)$ as $x \to \infty$
\[  \Phi_2(x) = \frac{1}{3x} + \frac{1}{\sqrt{27} x^{5/2}} + \mathcal{O}(1/x^4), \quad
    \Phi_3(x) = \frac{1}{3x} - \frac{1}{\sqrt{27} x^{5/2}} + \mathcal{O}(1/x^4), \qquad x \to \infty. \]

Recall that our $\Phi$ satisfies  $\Phi(x) = x^2-F(x)/3$, where $F(x) = \mathcal{O}(1/x)$, so that we need the solution $\Phi_1$. The discriminant of \eqref{eq:cubic} is $(4x^3-9)/27$ so that $\Phi_1$ has branch points at $(9/4)^{1/3}$, $(9/4)^{1/3} e^{\pm 2\pi i/3}$, which are three points on $\Gamma_0$, $\Gamma_1$, $\Gamma_2$ respectively, and since all the zeros of $P_{n,n}$ are on $\Gamma_0 \cup \Gamma_1 \cup \Gamma_2$, we conclude that the scaled zeros $x_{j,n}/n^{1/3}$ are dense on the three segments $[0,(9/4)^{1/3}] \cup [0,(9/4)^{1/3}e^{2\pi i/3}]\cup [0,(9/4)^{1/3}e^{-2\pi i/3}]$. 

The cubic equation can be solved explicitly by using Cardano's formula: let $y=x-1/(3\Phi)$ and $z=1/y$, then the cubic equation \eqref{eq:cubic} becomes
\[    z^3 -3xz +3 = 0, \]
and the solutions are $z=\omega^j u^{1/3}+\omega^{-j}v^{1/3}$ $(j=0,1,2)$, where $\omega=e^{2\pi i/3}$, $u+v=-3$ and $uv=x^3$, i.e.,
\[  u= \frac{-3 + \sqrt{9-4x^3}}{2}, \quad v = \frac{-3 - \sqrt{9-4x^3}}{2} = \frac{2x^3}{-3+\sqrt{9-4x^3}}. \]
The solution $\Phi_1$ corresponds to the solution with $z(x) = 1/x + \mathcal{O}(1/x^4)$, and this is $z(x) = \omega^2 u^{1/3} + \omega^{-2} v^{1/3}$, and since $\Phi = y^2$, we find
\[  \Phi_1(x) = \frac{1}{(\omega^2 u^{1/3} + \omega^{-2} v^{1/3})^2} = \frac{1}{\omega u^{2/3} + \omega^{-1} v^{2/3} +2x} .  \]
\end{proof}

\begin{corollary}  \label{cor:zeros}
Let $\{x_{j,n}, j=1,2,\ldots,2n\}$ be the zeros of $P_{n,n}$ and $\mu_{n}$ be the normalized counting measure of the scaled zeros $x_{j,n}/n^{1/3}$,
\[ \mu_{n} = \frac{1}{2n} \sum_{j=1}^{2n} \delta_{x_{j,n}/n^{1/3}}. \]
Then the sequence $(\mu_{n})_n$ converges weakly to the probability measure $\mu$ for which
\[   \int f(x)\, d\mu(x) = \int_0^{(9/4)^{1/3}} v(x) f(x)\, dx + \int_0^{\omega(9/4)^{1/3}} v(x) f(x)\, dx
+ \int_0^{\omega^2(9/4)^{1/3}} v(x) f(x)\, dx ,  \]
where $\omega=e^{2\pi i/3}$ and
\begin{equation}  \label{eq:density}
   v(x) =  \frac{\sqrt{3}}{4\pi} \Bigl( 1+x[a(x)+b(x)] \Bigr) [b(x)-a(x)],
\end{equation}
with
\begin{equation}  \label{eq:ab}
   a(x) = \left( \frac{3-\sqrt{9-4x^3}}{2} \right)^{1/3}, \quad b(x) = \left( \frac{3+\sqrt{9-4x^3}}{2} \right)^{1/3}.
\end{equation}
\end{corollary}

\begin{proof}
The Stieltjes transform of the measure $\mu_n$ is
\[  \int \frac{1}{x-t} \, d\mu_n(t) = \frac{1}{2n^{2/3}} \frac{P_{n,n}'(n^{1/3}x)}{P_{n,n}(n^{1/3}x)} , \]
hence Theorem \ref{thm:ratio} gives
\[  \lim_{n \to \infty} \int \frac{1}{x-t} \, d\mu_n(t) = \frac12 F(x) = \frac32 \Bigl( x^2-\Phi(x) \Bigr),  \]
uniformly on compact sets of $\mathbb{C} \setminus (\Gamma_0\cup\Gamma_1\cup\Gamma_2)$. The Grommer-Hamburger theorem
\cite{GH} then implies that $\mu_n$ converges weakly to a measure $\mu$ for which
\[   \int \frac{1}{x-t} \, d\mu(t) = \frac32 \Bigl( x^2-\Phi(x) \Bigr). \]
The function $\Phi$ is analytic in $\mathbb{C} \setminus ([0,(9/4)^{1/3}] \cup [0,\omega(9/4)^{1/3}] \cup [0,\omega^2(9/4)^{1/3}])$,
hence the measure $\mu$ is supported on $[0,(9/4)^{1/3}] \cup [0,\omega(9/4)^{1/3}] \cup [0,\omega^2(9/4)^{1/3}]$. Furthermore
it is absolutely continuous and we can find the density by using the Stieltjes inversion formula
\[   v(x) = - \frac{1}{\pi} \lim_{\epsilon \to 0+} \Im \frac{3}{2}  \Bigl( (x+i\epsilon)^2-\Phi(x+i\epsilon) \Bigr).  \]
Due to the $\omega$-symmetry, it is sufficient to determine $v(x)$ for $x \in [0,(9/4)^{1/3}]$. Clearly
\[   v(x) = \frac{3}{2\pi} \lim_{\epsilon \to 0+} \Im \Phi(x+i\epsilon) = \frac{3}{2\pi} \Im \frac{1}{(\omega^2a+\omega^{-2}b)^2}, \]
with $a$ and $b$ given in \eqref{eq:ab}. Then by some elementary (complex) calculus, using $(a^2-ab+b^2)(a+b)=a^3+b^3$ and $ab=x$, one
finds the expression \eqref{eq:density} for the density $v$.
\end{proof}

\begin{figure}[ht]
\begin{center}
\rotatebox{270}{\resizebox{!}{4in}{\includegraphics{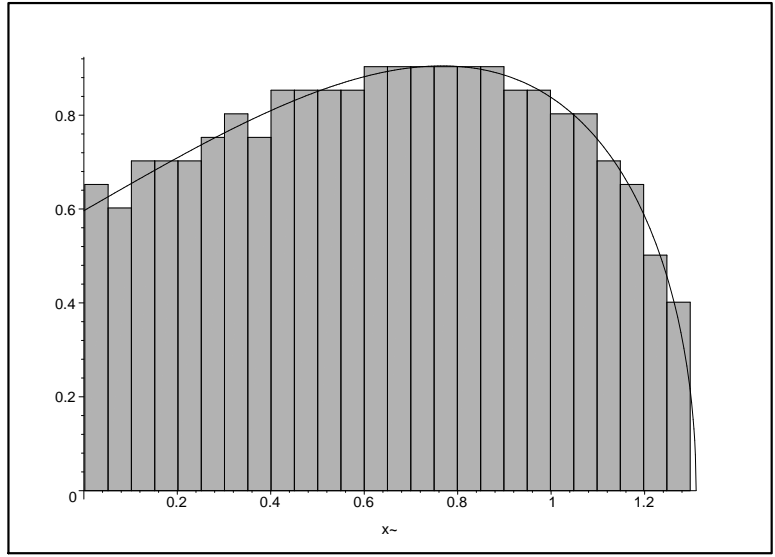}}}
\caption{Histogram of the real zeros of $P_{600,600}$ and the density $3v(x)$ on $[0,(9/4)^{1/3}]$}
\label{fig:6}
\end{center}
\end{figure}

In Figure \ref{fig:6} we have given a histogram of the 400 real zeros of $P_{600,600}$ together with the density $v$, scaled so as to have
total mass one for all the real zeros. There are 400 zeros on the interval $[0,\omega (9/4)^{1/3}]$ and 400 zeros on the interval
$[0,\omega^2 (9/4)^{1/3}]$ and these zeros are obtained by rotating the real zeros over an angle $\pm 2\pi/3$. The density $v$ has a finite non-zero
value at the origin $v(0)= 3^{1/3}\sqrt{3}/4\pi = 0.198788$ and tends to zero as $\sqrt{(9/4)^{1/3}-x}$ when $x\to (9/4)^{1/3}$.

If $k \neq l$, the rotational symmetry of the zeros is broken.
Suppose that $l \geq 2k$, then we see numerically that $k$ zeros of
$P_{k,l}$ lie on the line containing $\Gamma_1$ (some zeros are in
fact on $-\Gamma_1$), while the other zeros are distributed on a
complex contour in the lower half plane; see Figure \ref{fig:4}.
Similarly, if $k \geq 2l$, then $l$ zeros of $P_{k,l}$ lie on the
line containing $\Gamma_2$ (again some zeros are on $-\Gamma_2$),
and the other zeros are distributed on a complex contour in the
upper half plane, as illustrated in Figure \ref{fig:5}. Indeed, from
Theorem \ref{thm:Rodrigues}, it is easily seen that the zeros of
$P_{k,l}$ are complex conjugates of the zeros of $P_{l,k}$.
We expect that the asymptotic zero distribution of $P_{k,l}$ will
depend on the limit of the ratio $k/l$.

\begin{figure}[ht]
\begin{center}
 \resizebox*{7cm}{!}{\includegraphics{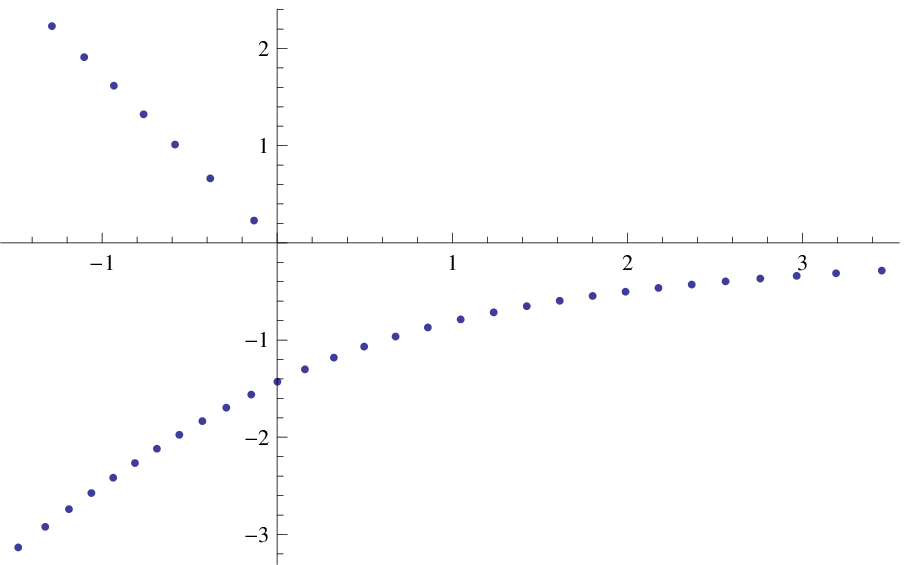}}
 \hspace{2mm} \resizebox*{7cm}{!}{\includegraphics{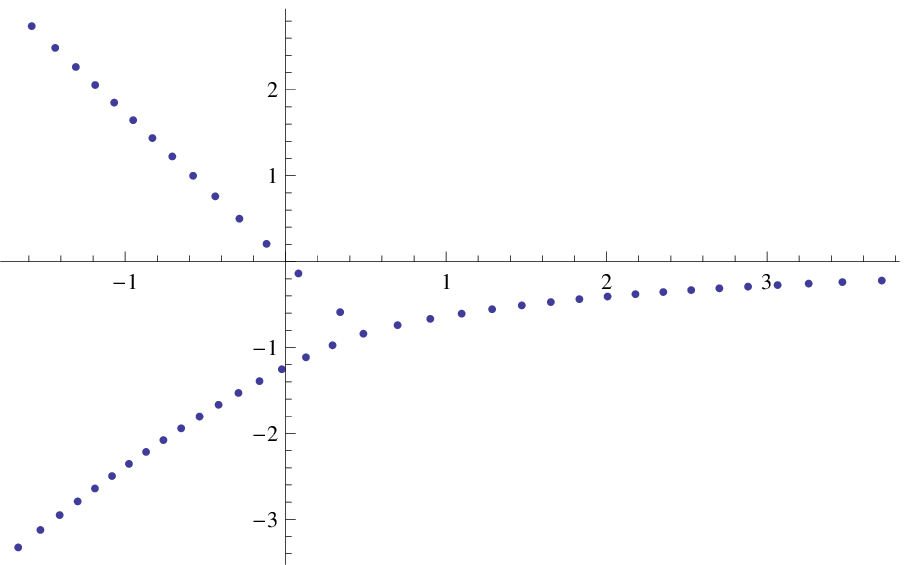}}
 \caption{Zeros of $P_{7,30}$ (left) and $P_{14,35}$ (right) after scaling} 
\label{fig:4}
\end{center}
\end{figure}

\begin{figure}[ht]
\begin{center}
 \resizebox*{7cm}{!}{\includegraphics{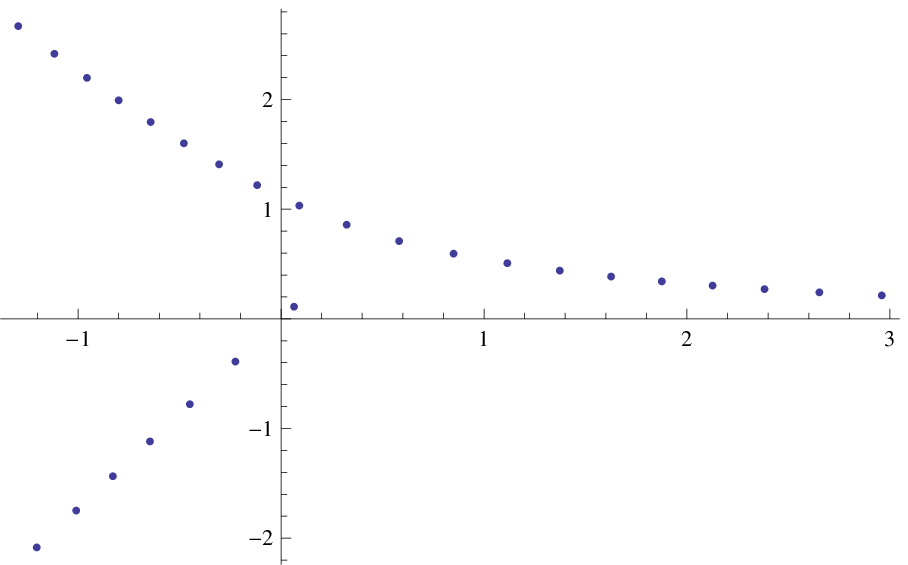}}
 \hspace{2mm} \resizebox*{7cm}{!}{\includegraphics{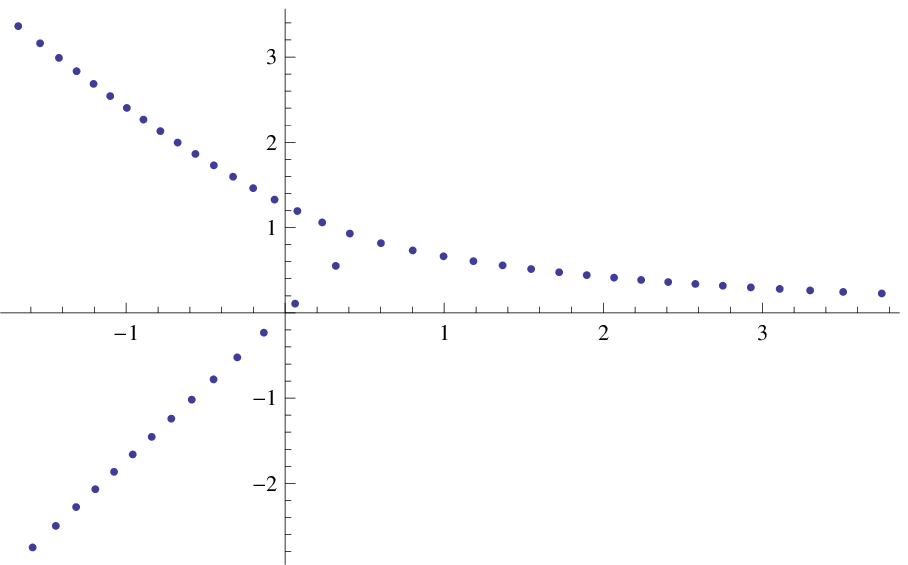}}
 \caption{Zeros of $P_{20,7}$ (left) and $P_{36,14}$ (right) after scaling} 
\label{fig:5}
\end{center}
\end{figure}


\section{Conclusions and outlook}
In this paper, we have introduced the multiple orthogonal
polynomials associated with an exponential cubic weight $e^{-x^3}$
over two contours in the complex plane. The basic properties of
these polynomials are studied, which include the Rodrigues formula
and nearest-neighbor recurrence relations. These results then allow
us to perform numerical studies of the recurrence coefficients and
zero distributions of the multiple orthogonal polynomials. Moreover, the recurrence coefficients are related to a discrete Painlev\'{e} equation. 
One can also consider the more general exponential cubic
weight $e^{-x^3+tx}$, where $t\in\mathbb{R}$, and the associated multiple
orthogonal polynomials have similar Rodrigues formulas and nearest-neighbor recurrence relations. 
Indeed, with $e^{-x^3}$ replaced by $e^{-x^3+tx}$ in Theorem \ref{thm:Rodrigues}, the difference-differential equations 
\eqref{eq:difference-diff eq} and \eqref{eq:differen-diff eq2} now read
\begin{align*}
P_{n,n+m}(x)&=(x^2-\frac{t}{3})P_{n-1,n+m-1}(x)-\frac{1}{3}P_{n-1,n+m-1}'(x), \\
P_{n+m,n}(x)&=(x^2-\frac{t}{3})P_{n+m-1,n-1}(x)-\frac{1}{3}P_{n+m-1,n-1}'(x).
\end{align*}
This then implies that 
\begin{align*}
\delta_{n,n+m}&=\delta_{0,m},\qquad \varepsilon_{n,n+m}=\varepsilon_{0,m}-\frac{t}{3}, \\
\delta_{n+m,n}&=\delta_{m,0},\qquad \varepsilon_{n+m,n}=\varepsilon_{m,0}-\frac{t}{3},
\end{align*}
where $\delta_{k,l}$ and $\varepsilon_{k,l}$ are defined in \eqref{eq:deltanm} and now depend on $t$. Following the same strategy as in the proof of Theorem \ref{thm:nearest recurrence} and using the string equation \eqref{eq:A string1} at the final stage, we obtain \eqref{eq:n+1} and \eqref{eq:n+m+1} with
\begin{align*}
c_{n,n+m}&=\left\{
            \begin{array}{ll}
              \beta_0^{(1)}(t), & \hbox{$m=0$,} \\
              -\beta_{m-1}^{(2)}(t), & \hbox{$m>0$,}
            \end{array}
          \right.
\\
d_{n,n+m}&=\beta_{m}^{(2)}(t),
\\
a_{n,n+m}&=\left\{
            \begin{array}{ll}
              \frac{n}{3\sqrt{3}}\frac{1}{\beta_0^{(1)}(t)-\beta_0^{(2)}(t)}, & \hbox{$m=0$,} \\
              -\frac{n}{m}(\gamma_m^{(2)}(t))^2-\frac{t}{3}, & \hbox{$m>0$,}
            \end{array}
          \right.
\\
b_{n,n+m}&=\left\{
            \begin{array}{ll}
              \frac{n}{3}\frac{1}{\beta_0^{(2)}(t)-\beta_0^{(1)}(t)}, & \hbox{$m=0$,} \\
              \frac{(n+m)}{m}(\gamma_m^{(2)}(t))^2, & \hbox{$m>0$,}
            \end{array}
          \right.
\end{align*}
where $\beta_n^{(1)}(t),(\gamma_{n}^{(1)})^2(t)$ are the recurrence coefficients of the monic orthogonal polynomials with respect to the weight 
$e^{-x^3+tx}$ on $\Gamma_0 \cup \Gamma_{1}$, and $\beta_n^{(2)}(t),(\gamma_{n}^{(2)})^2(t)$ are the recurrence coefficients of the monic orthogonal polynomials with respect to the weight $e^{-x^3+tx}$ on $\Gamma_0 \cup \Gamma_{2}$. The recurrence relations \eqref{eq:n+1 2} and \eqref{eq:n+m+1 2} for this general case can be obtained similarly, but we omit the results here. Clearly, in the general case, we lose the nice structure of the recurrence coefficients as stated in Theorem \ref{thm:nearest recurrence}, and more importantly we lose the symmetry given in Proposition \ref{prop:sym}, which
is why we focus on the weight $e^{-x^3}$ in this paper.


The challenging problem is to establish the asymptotic zero distribution of $P_{k,l}(x)$ for the non-symmetric case. At present we are unable to find an analogue of Theorem \ref{thm:ratio} and
Corollary \ref{cor:zeros} because of two reasons: first one needs the asymptotic behavior
of the recurrence coefficients and at present we can only conjecture the behavior
(see Proposition \ref{prop:anbn}). If we assume this to be correct, then the proof of
Theorem \ref{thm:ratio} can be used to find the asymptotic behavior, away from the set where
the zeros of the multiple orthogonal polynomials accumulate, in terms of an algebraic
function $\Phi$ satisfying a cubic equation. But the second reason is that we don't know where
the zeros of the multiple orthogonal polynomials accumulate. The discriminant of the cubic
equation is a quartic polynomial in $x$ and the four roots are branch points of the
algebraic function $\Phi$. The zeros will accumulate on two curves, each connecting two points in the complex plane, see Figures \ref{fig:4} and \ref{fig:5}. One of the curves is a straight line, the other is a curved line connecting two of the four branch points. The straight line, however, does not connect the other two branch points but starts from one branch point and stops before the second branch point is reached. This suggests that a vector equilibrium problem is involved, for two measures living on curves connecting four branch points, with an external field $x^3$ induced by the weight $e^{-x^3}$. In order to characterize the limiting zero distribution, one may need to extend the concept of $S$-property (cf. \cite{GoR,Stahl}) for orthogonal polynomials and equilibrium measures to this setting for multiple orthogonal polynomials and vector equilibrium problems. Once that is obtained, one may be able to use the Riemann-Hilbert method to find the asymptotic behavior of the multiple orthogonal polynomials.


\appendix
\section*{Appendix}

\section{Derivation of the string equations}
In this appendix, we give an alternative proof of the string
equations \eqref{eq:string 1}--\eqref{eq:string 2} using ladder
operators for orthogonal polynomials. Note that the ladder
operators for multiple orthogonal polynomials and their
compatibility conditions can be found in \cite{GVZ}.

Following the general set-up (cf. \cite{ci2}), if the weight
function $w$ vanishes at the endpoints of the orthogonality
interval, the lowering and raising ladder operators for the
associated monic polynomials $p_{n}$ are given by
\begin{align}
\left( \frac{d}{dx} + B_n(x) \right) p_n(x) & = \gamma_n^2 A_n(x)
p_{n-1}(x), \label{ladder1} \\
\left( \frac{d}{dx} - B_n(x) - \textsf{v}'(x) \right) p_{n-1}(x) & =
- A_{n-1}(x) p_n(x), \label{ladder2}
\end{align}
with
$$\textsf{v}(x):=-\ln w(x), $$
and
\begin{align} A_n(x) & :=
\frac{1}{h_n} \int \frac{\textsf{v}'(x) - \textsf{v}'(y)}{x-y} \
[p_n(y)]^2 w(y) dy, \label{an-def}\\ 
B_n(x) & := \frac{1}{h_{n-1}}
\int \frac{\textsf{v}'(x) - \textsf{v}'(y)}{x-y} \ p_{n-1}(y) p_n(y)
w(y) dy, \label{bn-def}
\end{align}
where
\begin{equation}\label{eq:orthogonality A}
\int p_m(x) p_n(x)\omega(x) dx = h_n \delta_{m,n}, \quad
m,n=0,1,2,\ldots.
\end{equation}
Note that $A_n$ and $B_n$ are not independent, but satisfy the
following compatibility conditions \cite[Lemma 3.2.2 and Theorem
3.2.4]{Ismail}.
\begin{proposition}
The functions $A_n$ and $B_n$ defined  in (\ref{an-def}) and
(\ref{bn-def}) satisfy 
$$
B_{n+1}(x) + B_n(x)  = (x- \beta_n) A_n(x) -
\textup{\textsf{v}}\,'(x), \eqno(S_1)
$$
$$
1+ (x- \beta_n) [B_{n+1}(x) - B_n(x)] = \gamma_{n+1}^2 A_{n+1}(x) -
\gamma_n^2 A_{n-1}(x).\eqno(S_2)
$$
\end{proposition}

Now we consider a more general exponential cubic weight
$e^{-x^3+tx}$, with parameter $t\in\mathbb{R}$. Then
$$\textsf{v}(x)=-\ln w(x)=x^3-tx ,$$ and
$$\frac{\textsf{v}\,'(x)-\textsf{v}\,'(y)}{x-y}=3(x+y).$$
It then follows from (\ref{an-def})--\eqref{eq:orthogonality A} that
\begin{equation}\label{eq:An Bn}
A_n(x)=3(x+\beta_n), \qquad B_n(x)=3\gamma_n^2.
\end{equation}
Substituting (\ref{eq:An Bn}) into $(S_1)$ and comparing the
coefficients of the constant term, we have
\begin{equation}\label{eq:A string1}
\gamma_n^2+\gamma_{n+1}^2+\beta_n^2-\frac{t}{3}=0.
\end{equation}
From $(S_2)$ we similarly get
\begin{equation}\label{eq:A string2}
3\gamma_n^2(\beta_{n-1}+\beta_n)=n.
\end{equation}
Note that in this case, the recurrence coefficients $\beta_n$ and
$\gamma_n^2$ all depend on $t$. By setting $t=0$ in \eqref{eq:A
string1} and \eqref{eq:A string2}, we recover the string equations
\eqref{eq:string 1} and \eqref{eq:string 2}.

The weight $e^{-x^3+tx}$ is a modification of the weight $e^{-x^3}$ with an exponential factor $e^{tx}$, and as a consequence
the recurrence coefficients satisfy the Toda equations \cite[\S 2.8]{Ismail}
\begin{align}
  \frac{d}{dt} \gamma_n^2 &= \gamma_n^2 ( \beta_n-\beta_{n-1}),  \label{Toda1} \\
  \frac{d}{dt} \beta_n &= \gamma_{n+1}^2-\gamma_n^2.  \label{Toda2}
\end{align}
If we differentiate \eqref{Toda2} and then use \eqref{Toda1}, we find
\[   \beta_n''(t) = \gamma_{n+1}^2 (\beta_{n+1}-\beta_n) - \gamma_n^2(\beta_n-\beta_{n-1}).  \]
Then use \eqref{eq:A string1} and \eqref{eq:A string2} to find
\[   \beta_n''(t) = 2\beta_n^3 -\frac{2t}{3}\beta_n + \frac{2n+1}{3}, \]
which is the Painlev\'e II equation. 
The equations \eqref{eq:A string1} and \eqref{eq:A string2} give
\[   \frac{n}{\beta_n+\beta_{n-1}} + \frac{n+1}{\beta_{n+1}+\beta_n} + 3\beta_n^2 = t  \]
which also follows from the B\"acklund transformation of the second Painlev\'e
equation (see \cite{FGR} and \cite{CMW}), hence it is not surprising to find that $\beta_n$ 
satisfies the Painlev\'e II equation.

If we write $x_n(t) = a\beta_n(-at)$, where $a=(3/2)^{1/3}$, then
\[    x_n''(t) = 2 x_n^3 + tx_n + n+ \frac12, \]
which is the Painlev\'e II equation in standard form and with parameter $\alpha=n+\frac12$. 
The second Painlev\'e equation is closely related to
the Airy equation and has special solutions in terms of Airy functions for the parameter values $\alpha = n +\frac12$, with $n \in \mathbb{Z}$
\cite[\S 7.1]{Clarkson}. This relation with the Airy function was to be expected since one has the integral representation
\[   \textrm{Ai}(t) = \frac{1}{2\pi i} \int_{\infty e^{-\pi i/3}}^{\infty e^{\pi i/3}} e^{z^3/3 - tz}\, dz, \]
(see Eq. 9.5.4 of the NIST Digital Library of Mathematical Functions\footnote{\texttt{http://dlmf.nist.gov}}) 
which contains (a slight variation of) the weight $e^{-z^3+tz}$. The special solution of Painlev\'e II in terms of Airy functions
is 
\[    x_n(t) = \frac{\tau_n'(t)}{\tau_n(t)} - \frac{\tau_{n+1}'(t)}{\tau_{n+1}(t)},  \]
where $\tau_n$ is the Hankel matrix
\[   \tau_n(t) = \begin{pmatrix} \varphi(t) & \varphi'(t) & \cdots & \varphi^{(n-1)}(t) \\
                                 \varphi'(t) & \varphi''(t) & \cdots  & \varphi^{(n)}(t) \\
                                    \vdots &   \vdots  &  \cdots & \vdots \\
                                  \varphi^{(n-1)}(t) & \varphi^{(n)}(t) & \cdots & \varphi^{(2n-2)}(t) 
               \end{pmatrix}, \]
and $\varphi$ is a solution of the Airy equation $\varphi''+\frac12 t\varphi = 0$. This solution of P$_{\textrm{II}}$ coincides with the well known solution
\[   \beta_n = \delta_n - \delta_{n+1}, \]
where $\delta_n$ is the coefficient of $x^{n-1}$ for the monic orthogonal polynomial $P_n(x)= x^n + \delta_nx^{n-1} + \cdots$. One has
$\delta_n = \Delta^*_n/\Delta_n$, where $\Delta_n$ is the Hankel matrix containing the moments
\[  \qquad \Delta_n = \begin{pmatrix} m_0 & m_1 & \cdots & m_{n-1} \\
                                                                               m_1 & m_2 & \cdots & m_{n} \\
                                                                               \vdots & \vdots & \cdots & \vdots \\
                                                                               m_{n-1} & m_n & \cdots & m_{2n-2} 
                                                              \end{pmatrix}  \]
and $\Delta^*_n$ is a similar determinant but with the last column replaced by $m_n, m_{n+1}, \ldots, m_{2n-1}$ respectively.
Even though this is an explicit solution, it is not convenient for finding the
recurrence coefficients when $n$ is large because of the high number of computations involved, whereas the relations \eqref{eq:A string1}--\eqref{eq:A string2} have a low computational complexity. The explicit solution is also not convenient for obtaining the asymptotic behavior of $\beta_n$ and $a_n^2$ as $n \to \infty$,
which is easier to obtain from the string equations (see Proposition \ref{prop:anbn}).

\section*{Acknowledgements}

The authors are grateful to the referees for their careful reading and constructive suggestions. WVA is supported by KU Leuven research grant OT/12/073,  FWO research grant G.0934.13 and the Belgian Interuniversity Attraction Poles Programme P7/18. GF is supported by the MNiSzW Iuventus Plus grant Nr 0124/IP3/2011/71. LZ was a Postdoctoral Fellow of FWO, and is also partially supported by The Program for Professor of Special Appointment
(Eastern Scholar) at Shanghai Institutions of Higher Learning (No. SHH1411007) and by Grant SGST 12DZ 2272800 from Fudan University.



\begin{thebibliography}{00}
\bibitem{Apt} A.I. Aptekarev,
\textit{Multiple orthogonal polynomials}, J. Comput.\ Appl.\ Math.
\textbf{99} (1998), 423--447.

\bibitem{AptBraWVA} A.I. Aptekarev, A. Branquinho and W. Van Assche,
\textit{Multiple orthogonal polynomials for classical weights},
Trans.\ Amer.\ Math.\ Soc.\ \textbf{355} (2003), 3887--3914.

\bibitem{BCBR} Y. Ben Cheikh and N. Ben Romdhane,
\textit{On $d$-symmetric classical $d$-orthogonal polynomials},
J. Comput.\ Appl.\ Math.\ \textbf{236} (2011), 85--93.

\bibitem{BD} P. Bleher and A. Dea\~{n}o,
\textit{Topological expansion in the cubic random matrix model}, 
Int. Math. Res. Not. \textbf{2013}, no. 12, 2699--2755.

\bibitem{ci2}  Y. Chen and M. Ismail,
\textit{Jacobi polynomials from compatibility conditions}, Proc.
Amer. Math. Soc. \textbf{133} (2005), 465--472.

\bibitem{Clarkson} P.A. Clarkson,
\textit{Painlev\'e equations --- Nonlinear special functions},
in ``Orthogonal Polynomials and Special Functions: Computation and Applications'' (F. Marcell\'an, W. Van Assche, Eds.), 
Lecture Notes in Mathematics \textbf{1883}, Springer-Verlag, Berlin, 2006, pp.~331--411.

\bibitem{CMW} P. Clarkson, E.L. Mansfield, H.N. Webster, 
\textit{On the relation between discrete and continuous Painlev\'e equations}, 
Theoret. and Math. Phys. \textbf{122} (2000), no. 1, 1--16.  

\bibitem{DHK} A. Dea\~{n}o, D. Huybrechs and A.B.J. Kuijlaars,
\textit{ Asymptotic zero distribution of complex orthogonal
polynomials associated with Gaussian quadrature}, J. Approx. Theory
\textbf{162} (2010), 2202--2224.

\bibitem{Dom} D. Dominici,
\textit{Asymptotic analysis of generalized Hermite polynomials},
Analysis \textbf{28} (2008), 239--261.

\bibitem{GVZ} G. Filipuk, W. Van Assche and L. Zhang,
\textit{Ladder operators and differential equations for multiple
orthogonal polynomials}, J. Phys. A: Math. Theor., \textbf{46}
(2013), 205204, 24 pp.

\bibitem{FGR} A.S. Fokas, B. Grammaticos and A. Ramani,
\textit{From continuous to discrete Painlev\'e equations}, J. Math.
Anal. Appl. \textbf{180} (1993), no. 2, 342--360.

\bibitem{FR} G. Freud,
\textit{On the coefficients in the recursion formulae of orthogonal
polynomials}, Proc. Royal Irish Acad. Sect. A \textbf{76} (1976),
1--6.

\bibitem{GH} J.S. Geronimo and T.P. Hill,
\textit{Necessary and sufficient condition that the limit of Stieltjes transforms is a Stieltjes transform},
 J. Approx. Theory \textbf{121} (2003), no. 1, 54--60.

\bibitem{GoR} A.A. Gonchar and E.A. Rakhmanov,
\textit{Equilibrium distributions and the rate of rational
approximation of analytic functions}, Mat. Sb. (N.S.)
\textbf{134(176)} (1987), 306--352, 447; translation in Math.
USSR-Sb. \textbf{62} (1989), no. 2, 305--348.

\bibitem{GouldHopper} H.W. Gould and A.T. Hopper,
\textit{Operational formulas connected with two generalizations of Hermite polynomials},
Duke Math.\ J. \textbf{29} (1962), 51--63.

\bibitem{GR} B. Grammaticos and A. Ramani,
\textit{The hunting for the discrete Painlev\'{e} equations}, Regul.
Chaotic Dyn. \textbf{5} (2000), 53--66.

\bibitem{GR2} B. Grammaticos and A. Ramani,
\textit{Discrete Painlev\'e equations: a review}, Lect. Notes Phys.
\textbf{644} (2004), 245--321.

\bibitem{Ismail} M.E.H. Ismail,
\textit{Classical and Quantum Orthogonal Polynomials in One
Variable}, Encyclopedia of Mathematics and its Applications
\textbf{98}, Cambridge University Press, 2005.

\bibitem{Magnus2} A.P. Magnus,
\textit{Painlev\'e type differential equations for the recurrence
coefficients of semi-classical orthogonal polynomials}, J. Comput.
Appl. Math. \textbf{57} (1995), 215--237.

\bibitem{NSKGR} F. Nijhoff, J. Satsuma, K. Kajiwara, B. Grammaticos and A. Ramani,
\textit{A study of the alternate discrete Painlev\'e II equation},
Inverse Problems \textbf{12} (1996),  697--716.

\bibitem{NikSor} E.M. Nikishin and V.N. Sorokin,
\textit{Rational Approximations and Orthogonality}, in: Translations
of Mathematical Monographs \textbf{92}, Amer.\ Math.\ Soc.\,
Providence RI, 1991.

\bibitem{Paris} R.B. Paris,
\textit{The asymptotics of the generalised Hermite-Bell polynomials},
J. Comput.\ Appl.\ Math.\ \textbf{232} (2009), 216--226.

\bibitem{Polya} G. P\'olya,
\textit{\"Uber die Nullstellen sukzessiver Derivierten},
Math.\ Z.\ \textbf{12} (1922), no.~1, 36--60.

\bibitem{PS} G. P\'olya and G. Szeg\H{o},
\textit{Problems and Theorems in Analysis II},
Springer-Verlag, Berlin, 1976 (revised and enlarged translation of
\textit{Ausgaben und Lehrs\"atze aus der Analysis II}, 4th edition, 1971).

\bibitem{Stahl} H. Stahl,
\textit{Orthogonal polynomials with complex-valued weight function.
I, II}, Constr. Approx. \textbf{2} (1986), 225--240, 241--251.


\bibitem{wva-C} W. Van Assche,
\textit{Discrete Painlev\'e equations for recurrence coefficients of
orthogonal polynomials}, in ``Difference Equations, Special
Functions and Orthogonal Polynomials" (S. Elaydi et al., Eds.),
World Scientific, 2007, pp. 687--725.

\bibitem{WVA} W. Van Assche,
\textit{Nearest neighbor recurrence relations for multiple
orthogonal polynomials}, J. Approx.\ Theory \textbf{163} (2011),
1427--1448.

\bibitem{WVAEC} W. Van Assche and E. Coussement,
\textit{Some classical multiple orthogonal polynomials}, Numerical
analysis 2000, Vol. V, Quadrature and orthogonal polynomials. J.
Comput.\ Appl.\ Math.\ \textbf{127} (2001), 317--347.

\end{thebibliography}
\end{document}